\documentclass[11pt,a4paper]{article}
\usepackage{amsthm}
\usepackage{amsfonts, amssymb, amsmath, amscd, latexsym}
\usepackage{eucal, enumerate, latexsym}
\usepackage{graphics}
\usepackage{graphicx}

\newcommand{\R}{\mathbb R}
\newcommand{\N}{\mathbb N}
\newtheorem{theorem}{Theorem}[section]

\newtheorem{proposition}{Proposition}[section]
\newtheorem{definition}{Definition}[section]
\newtheorem{remark}{Remark}[section]
\newtheorem{lemma}{Lemma}[section]

\newcommand{\orho}{\overline{\rho}}
\newcommand{\ds}{\displaystyle}

\renewcommand{\epsilon}{\varepsilon}
\newcommand{\eps}{\epsilon}
\renewcommand{\phi}{\varphi}
\newenvironment{proofof}[1]{\smallskip\noindent\emph{Proof of #1.}%
\hspace{1pt}}{\hspace{-5pt}{\nobreak\quad\nobreak\hfill\nobreak%
$\square$\vspace{8pt}\par}\smallskip\goodbreak}
\newcommand{\D}[1]{{\rm (D${}_{#1}$)}}
\makeatletter

\newlength{\captionwidth}
\long\def\@makecaption#1#2{%
   \vskip 10\p@
   \setbox\@tempboxa\hbox{#1: #2}%
   \ifdim \wd\@tempboxa > \captionwidth %\hsize
       \hbox to\hsize{\hfil
       \parbox[t]{\captionwidth}{
       \small#1: \small#2\par}
       \hfil}
     \else
       \hbox to\hsize{\hfil\box\@tempboxa\hfil}%
   \fi}

\makeatother
\begin{document}

\title{Semi-wavefront solutions in models of collective movements\\ with density-dependent diffusivity}

\author{Andrea Corli\\
\small\textit{Department of Mathematics and Computer Science, University of Ferrara}\\
\small\textit{I-44121 Italy, e-mail: andrea.corli@unife.it} \\
\and Luisa Malaguti\\
\small\textit{Department of Sciences and Methods for Engineering, University of Modena and Reggio Emilia}\\
\small\textit{I-42122 Italy, e-mail: luisa.malaguti@unimore.it}}

%\date{}
\maketitle
\begin{abstract} This paper deals with a nonhomogeneous scalar parabolic equation with possibly degenerate diffusion term; the process has only one stationary state. The equation can be interpreted as modeling collective movements (crowd dynamics, for instance). We first prove the existence of semi-wavefront solutions for every wave speed; their properties are investigated. Then, a family of travelling wave solutions is constructed by a suitable combination of the previous semi-wavefront solutions. Proofs exploit comparison-type techniques and are carried out in the case of one spatial variable; the extension to the general case is straightforward.

\vspace{1cm}

\end{abstract}

\section{Introduction}\label{s:I}
This paper deals with the scalar parabolic equation
\begin{equation}\label{e:E}
\rho_t + f(\rho)_x=\left( D(\rho)\rho_x\right)_x+g(\rho), \qquad (x,t)\in\R\times[0,+\infty),
\end{equation}
where $f\in C^1[0,\overline \rho]$, $f(0)=0$, $g\in C[0,\overline \rho]$ and $D\in C^1[0,\overline \rho]$, for some $\overline \rho >0$; we denote
\[
h(\rho)=f'(\rho).
\]
 The diffusion coefficient (or diffusivity) $D$ is required to satisfy one of the following assumptions, in increasing order of degeneracy at $0$:
\begin{itemize}

\item[{(D${}_0$)}] \, $D(\rho)>0$ for $\rho\in[0, \overline \rho]$;

\item[{(D${}_1$)}] \, $D(\rho)>0$ for $\rho\in(0, \overline \rho]$ and $D(0)=0$, $\dot D(0)>0$;

\item[{(D${}_2$)}] \, $D(\rho)>0$ for $\rho\in(0, \overline \rho]$ and $D(0)=\dot D(0)=0$.

\end{itemize}
We denoted by a dot the differentiation with respect to $\rho$. In the following, we simply refer to condition (D) when we indifferently assume either \D{0} or  \D{1} or else \D{2}. About the forcing term $g$ we require that it vanishes at $\overline{\rho}$, namely,
\begin{itemize}
\item[{(g)}] \, $g(\rho)>0$ for $\rho\in[0, \overline \rho)$ and $g(\overline \rho)=0$.
\end{itemize}

The reaction-diffusion-convection equation \eqref{e:E}, with $D$ vanishing (as a power function) at some points, models several physical and biological phenomena. We refer to \cite{GK, Kalashnikov, Murray, SanchezGarduno-Maini, Vazquez} for many applications and analytic results; however, none of these papers seems to deal with a source term $g$ satisfying (g) . The porous media equation \cite{Vazquez}, where $D(\rho)=m\rho^{m-1}$, does not enter in this framework if $1<m<2$, since in that case $D(0)=0$ but $\dot{D}(0)=\infty$; nevertheless, we shall provide results later on also for this case. However, our main source of inspiration has been the appearance of \eqref{e:E} with $g=0$ in the framework of collective movements, namely, traffic flows and crowd dynamics. We briefly account on this topic in the following lines.

\smallskip

The simplest continuum (macroscopic) model for traffic flow is probably the famous Lighthill-Whitham-Richards equation \cite{Lighthill-Whitham,Richards}
\begin{equation}\label{e:LWR}
\rho_t +\left(\rho v(\rho)\right)_x = 0.
\end{equation}
It coincides with \eqref{e:E} if $D=g=0$ and $f(\rho)= \rho v(\rho)$. Here, $\rho\in[0,\overline{\rho}]$ represents a density, $\overline{\rho}$ being the maximal density; the function $v(\rho)$ is an assigned speed, which is usually assumed to be decreasing and satisfying $v(\overline{\rho})=0$. Because of its simplicity, equation \eqref{e:LWR} is also the starting point for modeling crowd dynamics; we refer to \cite{Bellomo-Delitala-Coscia, Bellomo-Dogbe, Cristiani-Piccoli-Tosin,Rosinibook} for more information on these subjects.

Already Lighthill and Whitham \cite{Lighthill-Whitham} proposed to include a {\em linear} diffusion term in \eqref{e:LWR} to avoid the appearance of shock waves; in this case, the diffusivity $D$ is constant and \D{0} holds. We notice that the fundamental property of mass conservation, which clearly holds for equation \eqref{e:LWR}, is still valid in presence of a further diffusion term \cite[(3.43)]{Vazquez}. In recent years, several authors discussed the problem of choosing the \lq\lq correct\rq\rq\ diffusivity $D$. In particular, the paper \cite{Nelson_2000} (see also \cite{Nelson_2002,Payne}) considers the case
\begin{equation}\label{e:DNelson} D(\rho) = -\rho\left(Lv'(\rho) + \tau\rho\left(v'(\rho)\right)^2\right),
\end{equation}
where $L$ and $\tau$ are an anticipation distance and a relaxation time, respectively. Consider $v\in C^2[0,\overline{\rho}]$ and assume $\min_{\rho\in[0,\overline{\rho}]}\rho v'(\rho)>-L/\tau$, so that $D(\rho)>0$ if $\rho\in(0,\orho]$; if $v'(0)\not=0$ then \D{1} holds, if $v'(0)=0$ then \D{2} holds. On the other hand, (D) fails under the choice $v(\rho) = \min\left\{\overline{v},C\log(\overline{\rho}/\rho)\right\}$, where $\overline{v}$ is the maximal velocity and $C>0$ a constant \cite{Nelson_2000}: the problem is not only the loss of smoothness of $D$, which is discontinuous, but also the fact that it identically vanishes in a right neighborhood of $0$. A thorough discussion on the possible choices of $D$ is provided in \cite{Bellomo-Delitala-Coscia} and leads to discard constant diffusivities; the case $D(0)=D(\overline{\rho})=0$ is motivated in \cite{De_Angelis1999}, see also \cite{Bonzani2000, Bonzani-Mussone2002}. The case when both $h$ and $D$ depend on $x$ is considered in \cite{Burger-Karlsen_2003}.

Equation \eqref{e:E} also occurs in crowd dynamics, again in the case $g=0$. It has recently  been proposed in \cite{BTTV} (see also \cite{Cristiani-Piccoli-Tosin,Venuti-Bruno}) for
\begin{equation}\label{e:v1}
v(\rho)=\overline{v}\left( 1-{\rm e}^{-\gamma\left( \frac{1}{\rho}-\frac{1}{ \overline{\rho}}\right)}\right), \quad D(\rho) =-\delta \rho v'(\rho), \quad\rho \in [0,  \overline{\rho}].
\end{equation}
Here, $\overline{v}$ and $\overline{\rho}$ have the same meaning as above, $0<\gamma<\orho$ is obtained through experimental data and $\delta>0$ represents the characteristic depth of the visual field of pedestrians. In \eqref{e:v1} we clearly think of $v$ as a function defined in $(0,\overline{\rho}]$ and extended to $0$ in $C^\infty$ way. The choice \eqref{e:v1} satisfies assumption \D{2}; notice that $D(\overline{\rho})\ne0$. We emphasize that the exponential flatness of $D$ at $0$ due to \eqref{e:v1} is far from being common in applications; see however \cite[\S 21.3]{Vazquez}.

In the framework of collective movements, the case when $g$ does not vanish identically seems to have been often neglected but it is important to model entries or exits; we refer to \cite{Bagnerini-Colombo-Corli} for traffic flows and  \cite{Cristiani-Piccoli-Tosin} for pedestrian dynamics with zero diffusivity.
Usually such terms are localized in the space variables \cite{Bagnerini-Colombo-Corli}, but we chose both to a have a diffuse forcing and keep the assumptions on $g$ as simple as possible. Assumption (g) could be meaningful, for instance, in the case of pedestrians moving (or standing) along a long corridor (or street); if the number of side entries (cross streets, respectively) reaching the corridor is large, one could drop a model with many localized entries in favor of a model with a diffuse source term. Such a situation occurs, for instance, at the barriers of a subway exit; or where the platforms of a railway station reach the main hall; or replacing the corridor with a beach where the access is free. The assumption $g(\overline{\rho})=0$ in (g) models the fact that there is no room for further entries if the maximal density is reached. A simple example of forcing term $g$ satisfying (g) is $g(\rho) = L\cdot(\overline{\rho}-\rho)^\alpha$, for constants $L>0$ and $\alpha>0$: it plays a key role in Theorem \ref{t:strictly}.

In order to encompass all significative cases, in the last part of the paper we also consider the case when the slope of $D$ at $0$ is infinity, as it is the case for the above mentioned porous media equation if $1<m<2$. More precisely, in that part we assume $D\in C[0, \orho]\cap C^1(0, \orho)$ and one of the following conditions:
\begin{itemize}
\item[{({\^D}${}_0$)}] \, $D(\rho)>0$ for $\rho\in[0, \overline \rho]$ and $\dot D(0)=\pm \infty$.
    
\item[{({\^D}${}_1$)}] \, $D(\rho)>0$ for $\rho\in(0, \overline \rho]$ and $D(0)=0$, $\dot D(0)=\infty$.

\end{itemize}
As for (D), we simply refer to ({\^D}) when we assume indifferently either {({\^D}${}_0$)} or {({\^D}${}_1$)}. We treat case (D) first and separately from case ({\^D}) to avoid the discussion of several subcases at the same time; indeed, the techniques used for the former case are analogous to those exploited for the latter.

In this paper we consider neither the case where $D$ also vanishes at $\overline{\rho}$ nor, as a consequence, the case where in addition it changes sign in the interval $(0,\orho)$. In the case $g(0)=g(\orho)=0$, the former case was studied in \cite{MMconv}, the latter in \cite{Maini-Malaguti-Marcelli-Matucci}; see also \cite{Ferracuti-Marcelli-Papalini} for $D$ changing sign two times. Indeed, the case when $D$ changes sign is mentioned in \cite{Nelson_2000} and occurs in \eqref{e:DNelson} for particular but meaningful choices of $v$, $L$ and $\tau$; moreover, it naturally arises by applying the expansion of \cite{BTTV} to some velocity laws recently introduced in \cite{Colombo-Rosini2005} to model panic phenomena in crowd dynamics, see \cite{Rosinibook} for more information. The case when $D$ depends on $\rho_x$ has been studied by many authors, see for instance \cite{Garrione-Strani}.

Several extensions of the results provided in this paper, namely, the case when $D(\orho)=0$ and $g$ either satisfies (g) or can assume negative values, are contained in \cite{Corli-diRuvo-Malaguti}; here, we set up the main mathematical framework and only deal with the simplest application. Nevertheless, most of our results are new and are not contained in \cite{Ferracuti-Marcelli-Papalini,Maini-Malaguti-Marcelli-Matucci,MMconv}.

\smallskip

Now, we focus on the analytical aspects of the paper, that we believe are interesting by their own. A traveling-wave solution of equation \eqref{e:E} is a solution $\rho(x,t)$ satisfying $\rho(x,t)=\varphi(x-ct)$ for some wave profile $\varphi(\xi)$ and constant speed $c$. It is easy to see that  $\varphi(\xi)$ satisfies the equation
 \begin{equation}\label{e:tws}
\left(D(\varphi)\varphi^{\, \prime}\right)^{\, \prime} + \left(c-h(\varphi)\right)\varphi^{\, \prime}+g(\varphi)=0
  \end{equation}
in some open interval $I\subseteq \R$; we denoted by a prime the differentiation with respect to $\xi$. Since \eqref{e:tws} is an autonomous equation, every function  $\varphi(\xi-\xi_0)$, that is obtained from a traveling-wave solution by a shift of length $\xi_0$, is again a traveling-wave solution. Therefore, profiles can be unique only up to shifts. A traveling-wave solution between {\em two} stationary states of \eqref{e:E}, i.e. two zeros of $g$, which in addition is monotonic, is usually known as a \emph{wavefront solution}. For such sources $g$, equation \eqref{e:E} usually supports wavefront solutions; we refer to \cite{GK,GK05,MMconv} and references therein for recent results on this topic.

In \cite{DePablo-Sanchez} it is considered the case when $g$ only vanishes at {\em one} point (namely, $\rho=0$) and $D$, $h$, $g$ are polynomials. By classical techniques in the phase plane, the authors show the existence of global traveling wave solutions that decrease to $0$, see Section 2. According to our assumption (g), also equation \eqref{e:E} has only one stationary state (namely $\rho= \overline \rho$) but the domain of $g$ is the closed interval  $[0, \overline \rho]$; hence, only {\em semi-wavefront solutions}, see Section \ref{s:main}, may exist, as already showed in \cite{GK} in the special case $h(\varphi)\equiv 0$. Roughly speaking, the wave profiles of such solutions are only defined in a half-line $(-\infty,\varpi)$ or $(\varpi,+\infty)$ and tend to the stationary state either at $-\infty$ or to $+\infty$; a suitable change of variable commutes wave profiles of a type in those of the other type.

The aim of this paper is to extend some results of \cite{GK} to equation \eqref{e:E} when $h(\varphi)$ does not necessarily vanish identically, providing a unified treatment when either (D) or ({\^D}) holds; moreover, we improve the results in \cite{GK} by fully characterizing the slope of the wave profile when it reaches $0$. In Theorem \ref{t:semi}, which extends \cite[Theorem 6.1]{GK}, we prove that equation \eqref{e:E} has semi-wavefront solutions both with decreasing and with increasing profiles, for every wave speed $c$; moreover, such solutions are of class $C^2$ (see Remark \ref{rem:smoothphi}) and are unique up to shifts. We also explicitly compute the slope of the front when it reaches the value $0$, see Theorem \ref{t:semi}.
%However, even if semi-wavefront solutions $\rho(x,t)$ are important in some applications \cite{GK}, in the framework of collective movements their interest is limited by the fact that they are only defined in half-planes $x-ct\gtrless\varpi$: there is the need of considering globally defined solutions. This is the aim of Theorem \ref{t:2semi}, where we construct a wide class of solutions of \eqref{e:E} in $\R^2$ by a suitable combination of two semi-wavefront solutions. This construction has never been considered before, to the best of our knowledge, in the framework of diffusive equations, but it is common for dispersive equations \cite{Lenells_CH, Lenells_R}.
In Theorem \ref{t:strictly}, which parallels \cite[Theorem 6.2]{GK}, we fully characterize the semi-wavefront solutions of \eqref{e:E} that reach the value $\overline \rho$ only asymptotically: this happens if and only if $g$ satisfies condition \eqref{e:pend g}. A last result concerns the juxtaposition of two semi-wavefront profiles to obtain a global traveling-wave solution; while this procedure is succesfull for some dispersive equations \cite{Lenells_Classification}, we show that it is {\em not} effective in the current case.

The main technical tool to prove the above results is an order reduction of equation \eqref{e:tws}. Indeed, due to the sign condition on the source term $g$, it is possible to prove that every semi-wavefront solution has a wave profile $\varphi(\xi)$ that is strictly monotone in the region where $0 \le \phi(\xi) < \orho$, see Proposition \ref{p:monot}; hence, it is invertible there, with inverse function $\xi=\xi(\phi)$, $\phi\in[0,\orho)$. This allows us to reduce the second-order equation \eqref{e:tws} to a first-order equation; indeed, a straightforward computation shows that $z(\varphi):=D(\varphi)\varphi^{\, \prime}\left(\xi(\varphi)\right), \, \varphi\in (0, \overline \rho)$, satisfies the singular equation
\begin{equation}\label{e:zeq}
\dot z(\varphi)=h(\varphi)-c-\frac{D(\varphi)g(\varphi)}{z(\varphi)}, \qquad \varphi \in (0, \overline \rho).
\end{equation}
The study of \eqref{e:zeq} requires an original technique that has been developed in \cite{MMconv} and is  based on comparison-type arguments, i.e., on the existence of upper- and lower-solutions. The possible degenerate behavior of $D$ imposes a quite sharp construction of these solutions.

We mentioned above that assumption (D) fails if $D\in C^1[0,\orho]$ vanishes identically in $[0,\rho_1]$, for $0<\rho_1<\orho$, \cite{Nelson_2000}. However, if $D$ is strictly positive in $(\rho_1,\orho]$, then our results apply and provide wave profiles connecting $\rho_1$ with $\orho$. Moreover, our results directly extend to scalar parabolic equations in several space dimensions; in that case, the solutions are of the form $\rho(x,t) = \phi(x\cdot\nu - ct)$, where $\nu\in\R^n$, $|\nu|=1$, is a fixed vector and $x\in\R^n$. Indeed, in such a case equation \eqref{e:tws} becomes
\[
\left(D(\varphi)\varphi^{\, \prime}\right)^{\, \prime} + \left(c-h(\varphi)\cdot\nu\right)\varphi^{\, \prime}+g(\varphi)=0,
\]
which is analogous to \eqref{e:tws}.

The plan of the paper now follows. Section \ref{s:main} contains the statements of the main results; proofs are postponed to the following sections, in particular to Section \ref{s:proofs}. Section \ref{s:example} shows some applications to the model for crowds dynamics introduced above. Sections \ref{s:comp} to \ref{sec:pasting} deal with case (D). In Section \ref{s:comp} we prove some basic facts about equation \eqref{e:zeq}; the study of a first-order boundary value problem related to that equation is completed in Section \ref{s:first solution} while in Section  \ref{s:equiv} we show the equivalence of \eqref{e:tws} and \eqref{e:zeq}. In the final Section \ref{sec:pasting} we discuss the problem of pasting semi-wavefront profiles to obtain global traveling-wave solutions. Case ({\^D}) is studied in Section \ref{sec:slopeinfty}.
%%%

\section{Main results}\label{s:main}
\setcounter{equation}{0}
This section contains the main results of the paper. First, we introduce the notions of traveling-wave and semi-wavefront solutions to \eqref{e:E}; assumptions (D) and (g) are not needed in these definitions. We refer to \cite{GK} for more details.

\begin{definition}\label{d:tws} Let $I\subseteq \R$ be an open interval, $c\in\R$ and $\varphi \colon I \to [0, \overline{\rho}]$ such that $\varphi\in C(I)$, $D(\phi)\varphi^{\, \prime}\in L_{\rm loc}^1(I)$ and
\begin{equation}\label{e:def-tw}
\int_I \left\{\left(D\left(\phi(\xi)\right)\phi'(\xi) - f\left(\phi(\xi)\right) + c\phi(\xi) \right)\psi'(\xi) - g\left(\phi(\xi)\right)\psi(\xi)\right\}\,d\xi =0,
\end{equation}
for every $\psi\in C_0^\infty(I)$. Then, for all $(x,t)$ with $x-ct \in I$, the function $\rho(x,t)=\varphi(x-ct)$ is said a traveling-wave solution of equation \eqref{e:E} with wave speed $c$ and wave profile $\phi$. The traveling-wave solution is global if $I=\R$.
\end{definition}

If $\phi$ is a differentiable function, $\varphi'$ is absolutely continuous and \eqref{e:tws} holds a.e., then clearly \eqref{e:def-tw} holds. Such profiles, as well as the corresponding traveling-wave solutions, are called {\em classical}. In this paper we always deal with classical profiles; nevertheless, the above definition of {\em weak} solution is exploited in Section \ref{sec:pasting}. %Notice that Definition \ref{d:tws} slightly differs from that given in \cite{GK}, which requires $D(\phi)\phi'\in L^1_{\rm loc}(I)$: indeed, under our assumption we have that $D(\phi)\in L^\infty(I)$.

Of course, if $\phi(\xi)$ satisfies \eqref{e:def-tw} (or \eqref{e:tws} a.e. in $I$), then $\rho(x,t)=\varphi(x-ct)$ is a weak solution of (resp., solves a.e.) \eqref{e:E} in the corresponding subset of $\R^2$.

Now, we introduce semi-wavefront solutions for equation \eqref{e:E}. With respect to traveling-wave solutions, we essentially require that the wave profiles are defined in a half-line and, as a consequence, tend to a stationary value either at $+\infty$ or at $-\infty$.

\begin{definition}\label{d:swf}
Consider a traveling-wave solution $\rho$ of equation \eqref{e:E} whose wave profile $\phi$ is defined in $(\varpi, +\infty)$, with $\varpi \in \R$; let $\ell^+\in[0,\overline{\rho}]$ be such that $g(\ell^+)=0$. If $\phi$ is monotonic, non-constant and
$$
\varphi(\xi) \to \ell^+ \quad \text{as } \xi \to +\infty,
$$
then $\rho$ is said a \emph{semi-wavefront solution} of \eqref{e:E} {\em to} $\ell^+$.

Similarly, assume that $\phi$ is defined $(-\infty, \varpi)$ and let $\ell^-\in[0,\overline{\rho}]$ be such that $g(\ell^-)=0$. If $\phi$ is monotonic, non-constant and
$$
\varphi(\xi) \to \ell^- \quad \text{as } \xi \to -\infty,
$$
then $\rho$ is said a \emph{semi-wavefront solution} of \eqref{e:E} \emph{from} $\ell^-$.

In both cases, a semi-wavefront solution is {\em strict} if it is not extendible to a global traveling-wave solution.
\end{definition}

For sake of precision, we point out that {\em monotonic} in the previous definition is meant in the weak sense: if $\xi_1< \xi_2$, then either $\phi(\xi_1)\le \phi(\xi_2)$ or $\phi(\xi_1)\ge \phi(\xi_2)$. Above and in the following, wave profiles are always defined in their {\em maximal} existence interval. Due to the regularity of $D$ and $g$, we will show in the following (see Theorem \ref{t:semi}) that \eqref{e:E} always admits classical semi-wavefront solutions for any $c\in\R$.

For comparison with our results, we first provide a simple application of \cite{GK} to equation \eqref{e:E} {\em in the case } $g\equiv0$, namely:
\begin{equation}\label{e:Ewg}
\rho_t+f(\rho)_x=\left( D(\rho)\rho_x\right)_x, \qquad (x,t)\in\R\times[0,+\infty).
\end{equation}
We notice that any $\tilde\rho\in[0,\orho]$ is an equilibrium for equation \eqref{e:Ewg}; however, for simplicity, we only focus on $\orho$. 

\begin{theorem}\label{t:swfwg} Assume condition {\rm (D)}; then
%, about the existence of semi-wavefront solutions to equation \eqref{e:Ewg}, with wave speed $c$, 
we have the following results.

\begin{enumerate}[(i)]
\item If $c <h(\orho)$ ($c > h(\orho)$), then equation \eqref{e:Ewg} has a classical semi-wavefront solution from $\orho$ (resp., to $\orho$) with wave speed $c$;

\item if $c = h(\orho)$, the same result holds if and only if for some $0<\delta \le \orho$ we have
\begin{equation*}
\int_{\orho-s}^{\orho}\left[h( \sigma)- h(\orho)\right]\, d\sigma > 0 \ \left(\hbox{ resp.,}\int_{\orho-s}^{\orho}\left[h( \sigma)- h(\orho)\right]\, d\sigma < 0\right),\quad \text{for }0<s<\delta;
\end{equation*}

\item if $c > h(\orho)$ ($c < h(\orho)$), then equation \eqref{e:Ewg} has no classical semi-wavefront solution from $\orho$ (resp., to $\orho$) with wave speed $c$.
\end{enumerate}
Moreover, when the above semi-wavefront solutions exist they are unique up to shifts and their wave profiles are of class $C^2$ in $(-\infty,\varpi)$ or $(\varpi,+\infty)$, respectively.
\end{theorem}
Condition {\em (ii)} in Theorem \ref{t:swfwg} is satisfied if $f$ is strictly concave (resp., strictly convex) in a neighborhood of $\orho$. For sake of completeness, in the case $f$ is strictly concave we rephrase \cite[Theorem 9.1]{GK}, which concerns wavefront solutions \cite[p. 5]{GK}.

\begin{proposition}\label{p:wfTosinI}
Consider equation \eqref{e:Ewg} under assumption {\rm (D)}, where $f$ is strictly concave; fix $\rho^-\in[0,\orho]$ and $\rho^+\in[0,\orho]$, $\rho^-\ne \rho^+$. Then, wavefront solutions connecting $\rho^-$ with $\rho^+$ exist if and only if $\rho^-<\rho^+$; in that case we have
\[
c=\frac{f(\rho^+)-f(\rho^-)}{\rho^+-\rho^-}.
\]
If $\rho_->0$, wavefront solutions are classical and strictly monotonic; if $\rho_- = 0$, wavefront solutions are still classical and strictly monotonic under \D{0} but they are weak under both \D{1} or \D{2}. In the latter case, we have $\phi(\xi)=0$ for $\xi\in(-\infty,\xi_0]$, for some $\xi_0\in\R$, with $\phi'(\xi_0^+)>0$ under \D{1} and $\phi(\xi_0^+)=\infty$ under \D{2}.
\end{proposition}

From now on we only deal with the full equation \eqref{e:E}. For brevity we often provide complete statements and proofs for semi-wavefront solutions from $\orho$; analogous results hold for semi-wavefront solutions to $\orho$. Two results on the existence of semi-wavefront solutions follow. 

\begin{theorem} \label{t:semiequiv}  \ Assume {\rm (D)} and {\rm (g)}. The existence of a strict semi-wavefront solution from $\overline \rho$ of  \eqref{e:E} with speed $c$ is equivalent to the solvability, for the same $c$, of the boundary-value problem
\begin{equation}\label{e:fo}
\left\{
\begin{array}{l}
\dot z = h(\varphi)-c-\frac{D(\varphi)g(\varphi)}{z}, \\
z(\varphi)<0,  \quad \varphi \in (0,\overline \rho),\\
z(0\,^+)=:z_0 \le 0, \quad z(\overline{\rho}\,^-)=0.
\end{array}
\right.
\end{equation}
\end{theorem}

Since the first equation in \eqref{e:fo} is singular and its right-hand side is not defined at points $\phi_0$ where $z(\phi_0)=0$, we used the limit notation $z(\phi_0^\pm)$ for such points. We emphasize that $z_0$ is {\em not} a datum in \eqref{e:fo} but simply a shortcut for the real number $z(\,0^+)$. Moreover, the requirement $z(\varphi)<0$ for $\varphi \in (0,\overline \rho)$ is always satisfied if $z_0<0$ and then it is needed only when $z_0=0$. Solutions $z$ to \eqref{e:fo} are meant in the sense $z\in C^0[0,\overline{\rho}]\cap C^1(0,\overline{\rho})$.

The next theorem extends to the case $z(0\,^+)<0$ an analogous result proved in \cite{MMconv} in the case $z(\,0^+)=0$ and under \D{2}.

\begin{theorem} \label{t:fosemi} Assume {\rm (D)} and {\rm (g)}. Then, problem \eqref{e:fo} is uniquely solvable for every  $c\in \mathbb{R}$. More precisely, in case {\rm \D{0}} we have $z(0^+)<0$ for every $c$; in cases {\rm \D{1}} and {\rm \D{2}} there exists a real number $c^*$ satisfying the estimate
\begin{equation}\label{e:cs}
2\sqrt{\dot D(0)g(0)}+h(0)\ \le\ c^*\le\ 2\sqrt{\sup_{s\in (0, \overline \rho)}\frac{D(s)g(s)}{s}}\ +\max_{\rho \in [0, \overline \rho]}h(\rho),
\end{equation}
such that $z(0^+)<0$ if $c<c^*$ and $z(0^+)=0$ if $c\ge c^*$.
%\eqref{e:fo} with $z(\,0^+)=0$ is solvable if and only if $c\ge c^*$.
\end{theorem}

We notice that in case \D{2} the inequalities in \eqref{e:cs} reduce to
\begin{equation}\label{e:threshold}
h(0)\ \le\ c^*\le\ 2\sqrt{\sup_{s\in (0, \overline \rho)}\frac{D(s)g(s)}{s}}\ +\max_{\rho \in [0, \overline \rho]}h(\rho).
\end{equation}

Now, we provide our main results. The first one concerns the existence of strict semi-wavefront solutions to \eqref{e:E}, under assumptions (D) and (g); since wave profiles are defined in their maximal existence interval, then $\phi(\varpi)=0$; see Figure \ref{f:SWFs}. For brevity, in cases \D{1}, \D{2} and for $c\ge c^*$, we introduce the notation
\begin{equation}\label{e:r_pm}
r_\pm(c) := \frac{h(0)-c\pm \sqrt{(h(0)-c)^2-4\dot D(0)g(0)}}{2}.
\end{equation}
Notice that the term under square root is positive because of \eqref{e:cs}; moreover, $r_\pm(c)<0$.

\begin{theorem}[Semi-wavefront solutions]\label{t:semi} Consider equation \eqref{e:E} under assumptions  {\rm (D)} and {\rm (g)}. Then, the following holds.

\begin{enumerate}[{(i)}]
\item For every wave speed $c \in \mathbb{R}$, equation \eqref{e:E} has a strict classical semi-wavefront solution from $\overline \rho$ and a strict classical semi-wavefront solution to $\overline \rho$. These solutions are unique up to shifts and their wave profiles are of class $C^2$ in $(-\infty,\varpi)$ or $(\varpi,+\infty)$, respectively.

\item Consider a semi-wavefront solution from $\orho$; then, about the slope of the profile when it reaches $0$, we have:
\begin{align}
\hbox{in case {\rm \D{0}}: \quad} & \lim_{\xi \to \varpi^-}\varphi^{\, \prime}(\xi)\in(-\infty,0),\label{e:slope0}
\\
\hbox{ in case {\rm \D{1}}: \quad} &
\lim_{\xi \to \varpi^-}\varphi^{\, \prime}(\xi)=
\left\{
\begin{array}{ll}
-\infty & \hbox{ if } c < c^*,\\
\ds\frac{r_-(c^*)}{\dot D(0)} & \hbox{ if } c = c^*,\\
\ds\frac{r_+(c)}{\dot D(0)} & \hbox{ if } c>c^*,
\end{array}
\right.
\label{e:slope1}
\\
\hbox{ in case {\rm \D{2}}: \quad} & \lim_{\xi \to \varpi^-}\varphi^{\, \prime}(\xi)=
\left\{
\begin{array}{ll}
-\infty & \hbox{ if } c \le c^*,\\
-\ds\frac{g(0)}{c-h(0)} & \hbox{ if } c>c^*.
\end{array}
\right.
\label{e:slope2}
\end{align}

\item Let $\phi_1$ and $\phi_2$ be two profiles corresponding to semi-wavefront solutions from $\orho$ with wave speeds $c_1 < c_2$, respectively; unless of a shift we can assume $\varpi_1=\varpi_2=:\varpi$. Then
\begin{equation}\label{e:remark}
\phi_2(\xi) < \phi_1(\xi), \quad \text{for } \xi\in(-\infty,\varpi) \text{ with }\phi_2(\xi)<\orho.
 \end{equation}

\end{enumerate}
In cases (ii) and (iii) analogous results hold for semi-wave-front solutions to $\orho$.
\end{theorem}

%%%%%%%%%%%%%%%%%%%%%%%% Figure SWFs

\begin{figure}[htbp]
\begin{picture}(100,100)(80,-10)
\setlength{\unitlength}{1pt}

\put(160,0){
\put(0,0){\vector(1,0){260}}
\put(0,60){\line(1,0){260}}
\put(260,8){\makebox(0,0){$\xi$}}
\put(130,-10){\vector(0,1){100}}
\put(137,87){\makebox(0,0){$\phi$}}
\put(137,67){\makebox(0,0){$\overline{\rho}$}}

\put(0,0){\thicklines{\qbezier(50,60)(135,60)(150,0)}}
\put(50,60){\thicklines{\line(-1,0){50}}}
\put(100,42){\makebox(0,0)[b]{$\phi_3$}}
\put(30,62){\makebox(0,0)[b]{$\phi_3$}}
\put(150,-2){\makebox(0,0)[t]{$\varpi_3$}}
\put(125.5,40){\thicklines{\vector(3,-2){3}}}

\put(0,0){\thicklines{\qbezier(0,56)(70,56)(110,0)}}
\put(60,32){\makebox(0,0)[b]{$\phi_1$}}
\put(110,-2){\makebox(0,0)[t]{$\varpi_1$}}
\put(83,29){\thicklines{\vector(1,-1){3}}}

\put(0,0){\thicklines{\qbezier(160,0)(160,58)(260,58)}}
\put(220,43){\makebox(0,0)[b]{$\phi_2$}}
\put(164,-2){\makebox(0,0)[t]{$\varpi_2$}}
\put(191,46){\thicklines{\vector(2,1){3}}}

%\multiput(159,0)(0,5){6}{$.$}
%\put(160,-2){\makebox(0,0)[t]{$\xi_0$}}
%
%\multiput(130,29)(5,0){6}{$.$}
%\put(128,30){\makebox(0,0)[r]{$\rho_0$}}
}
\end{picture}
\caption{\label{f:SWFs}{A strictly decreasing semi-wavefront solution $\phi_1$ from $\overline{\rho}$, a strictly increasing semi-wavefront solution $\phi_2$ to $\overline{\rho}$, a non-strictly decreasing semi-wavefront solution $\phi_3$ from $\overline{\rho}$.}}
\end{figure}
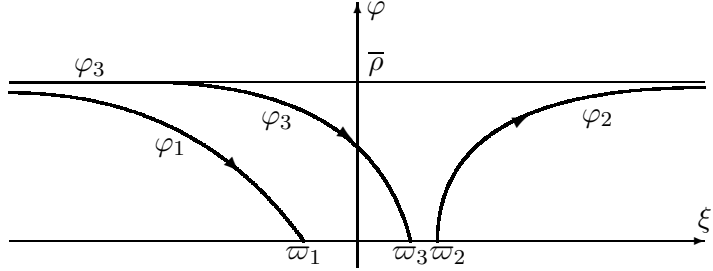
%%%%%%%%%%%%%%%%%%%%%%%%%%%%% End Figure SWFs

\begin{remark}{\rm Theorem \ref{t:semi} states that under assumption (g) we have semi-wavefront solutions of speed $c$ {\em for every value of} $c$.

However, in the case $g=0$ we only have such solutions {\em for some values of} $c$, see Theorem \ref{t:swfwg}: assuming for instance $h\equiv0$, then semi-wavefront solutions from $\orho$ may only move to the {\em left} and semi-wavefront solutions to $\orho$ may only move to the {\em right}. This is the effect of diffusion, which spreads the wave and then makes the function $t\to\rho(x,t)$ {\em decrease}.

On the contrary, in the case when $D\equiv0$, one expects that the presence of the positive source term $g$ makes the function $t\to\rho(x,t)$ {\em increase}; as a consequence, semi-wavefront solutions from $\orho$ should move to the {\em right} and semi-wavefront solutions to $\orho$ should move to the {\em left}.

In presence of both diffusion and source term, these opposite behaviors tune up and lead to the existence of semi-wavefront solutions for every $c$. The reader can convince her/himself of this tuning in the case $h\equiv0$, $D$ constant and $g(\rho) = 1-\rho$, where explicit solutions are easily constructed.
}
\end{remark}

As depicted in Figure \ref{f:SWFs}, the wave profiles can reach the value $\orho$ for a finite $\xi_0$ and then assume identically the value $\orho$ for $\xi<\xi_0$ (or $\xi>\xi_0$). %Because of Definition \ref{d:tws}, the statement of Theorem \ref{t:semi} implicitly includes the fact that such wave profiles are differentiable at $\xi_0$.

Let $\rho$ be any semi-wavefront solution in Theorem \ref{t:semi} and $\phi$ its wave profile; assume that $\phi$ is defined either in $(-\infty,\varpi)$ or in $(\varpi,+\infty)$. Because of (g), the value $\phi(\varpi)$ is not an equilibrium of \eqref{e:E} (semi-wavefront solutions are strict) and, as a consequence, the value $\phi(\varpi)$ is not a constant solution of \eqref{e:E}. This is a striking difference with the applications considered in \cite{GK}, where semi-wavefront solutions, when they are considered for equations without source terms, have $\phi(\varpi)$ as a solution of the equation.

Definition \ref{d:swf} requires that semi-wavefront solutions possess monotonic wave profiles; clearly, in the statement of Theorem \ref{t:semi} profiles from $\overline{\rho}$ are decreasing while profiles to $\overline{\rho}$ are increasing. Our next result shows that non-strictly monotonic wave profiles, such as $\phi_3$ in Figure \ref{f:SWFs}, can be ruled out by requiring a growth condition on the source term $g$ in a neighborhood of $\orho$; indeed, such condition is sharp.

\begin{theorem}[Characterization of strictly monotonic solutions]\label{t:strictly} Consider equation \eqref{e:E} under assumptions  {\rm (D)} and {\rm (g)}; let $L>0$ and $\rho_1\in[0,\orho)$ be two constants.
 \begin{itemize}
 \item[{(i)}] \ If
\begin{equation}\label{e:pend g}
g(\rho)\le L(\overline \rho-\rho), \qquad \rho \in [\rho_1, \overline \rho],
\end{equation}
 then the wave profile $\varphi$ of every semi-wavefront solution satisfies $\varphi(\xi)<\overline \rho$ for every $\xi$ in its domain.

 \item[{(ii)}] \ If there is $\alpha \in (0,1)$ such that
\begin{equation}\label{e:g sub}
g(\rho)\ge L(\overline \rho-\rho)^{\alpha}, \qquad \rho \in [\rho_1, \overline \rho],
\end{equation}
 then every wave profile $\varphi$ of a semi-wavefront solution satisfies $\varphi(\xi)\equiv \overline \rho$ on $(-\infty, \overline \xi]$ (or on $[\overline \xi,+\infty)$), for some $\overline \xi$ in its domain.
\end{itemize}
\end{theorem}

In our last result we only require $D\in C[0, \orho]\cap C^1(0, \orho)$ and assume ({\^D}); this means that we allow $D$ to have infinite slope at $0$. The statement below is analogous to those of Theorems \ref{t:fosemi} and \ref{t:semi}; notice that under ({\^D}${}_1$) we formally deduce $c^*=\infty$ in \eqref{e:cs}, which would suggest the solvability of \eqref{e:fo} for any $c$ and, moreover, $z(0^+)<0$. This is indeed the case. For brevity we only deal with the case of profiles from $\orho$; the case of profiles to $\orho$ is completely analogous.

\begin{theorem}\label{t:Dinf}
Assume {\rm ({\^D}${}$)} and {\rm (g)}; then, problem \eqref{e:fo} is uniquely solvable for every $c\in\R$ and $z(0^+)<0$ for every $c$. In turn, equation \eqref{e:E} has a strict classical semi-wavefront solution from $\overline \rho$ for every $c$; solutions are unique up to shifts and their wave profiles are of class $C^2$ in $(-\infty,\varpi)$. Moreover, 
\begin{equation}\label{e:lastissimo}
\hbox{in case {\rm ({\^D}${}_0$)}:}  \lim_{\xi \to \varpi^-}\varphi^{\, \prime}(\xi)\in(-\infty,0),
\quad
\hbox{ in case {\rm ({\^D}${}_1$)}:} 
\lim_{\xi \to \varpi^-}\varphi^{\, \prime}(\xi)=-\infty.
\end{equation}
\end{theorem}

Results analogous to those stated in Theorem \ref{t:semi}{\em (iii)} and Theorem \ref{t:strictly} still hold under assumption {\rm ({\^D}${}$)}.

For simplicity, in the following sections we shorten the expression \lq\lq $\rho$ is a semi-wavefront solution of \eqref{e:E} from $\overline{\rho}$ with wave profile $\phi$\rq\rq\ by writing \lq\lq $\phi$ is a semi-wavefront of \eqref{e:E} from $\overline{\rho}$\,\rq\rq\ and so on.

%%%
\section{An example}\label{s:example}
\setcounter{equation}{0}

Consider the model of crowd dynamics discussed in the Introduction, namely
\begin{equation}\label{e:model_Tosin}
\rho_t + \left(\rho v(\rho)\right)_x = \left(D(\rho)\rho_x\right)_x.
\end{equation}
We assume, as it is often usual in this modeling, that $f(\rho) = \rho v(\rho)$ is a strictly concave function; in particular this assumption is satisfied if $v$ and $D$ are given by \eqref{e:v1}. In such a case, $h(\rho) =v(\rho)+\rho v^{\, \prime}(\rho)$ and $h(\overline{\rho})=-\frac{\gamma \overline v}{\overline{\rho}}$; as in Theorem \ref{t:swfwg}, we only focus on $\orho$. If (D) holds, by Theorem \ref{t:swfwg} and the comment following it we deduce that equation \eqref{e:model_Tosin} has
\begin{enumerate}[{\em (i)}]

\item a semi-wavefront solution from $\overline \rho$ (to $\overline \rho$) for every $c \le h(\overline{\rho})$ (resp., $c > h(\overline{\rho})$);

\item no semi-wavefront solution from $\overline \rho$ (to $\overline \rho$) if $c > h(\overline{\rho})$ (resp., $c \le h(\overline{\rho})$).
\end{enumerate}
In all cases the corresponding profile $\phi$ is a solution in a half-line $I$ of
\begin{equation}\label{e:CD}
\left(D(\phi)\phi^{\prime}\right)^{\prime}+\left(c\phi-\phi v(\phi) \right)^{\prime}=0.
\end{equation}
Moreover, by \cite[Theorem 5.2]{GK} we have that $\phi(\xi)\in(0,\orho)$ for every $\xi\in(-\infty,\varpi)$.
 
Now, we show some additional results about \eqref{e:model_Tosin}. 

\begin{lemma}\label{l:Tosin1}
Let $c\le h(\orho)$ and $\phi$ be a classical semi-wavefront profile from $\orho$ for \eqref{e:model_Tosin}. Then, there exists $\varpi\in\R$ such that $\varphi(\xi)\to 0$ as $\xi\to\varpi^-$; moreover,
 \begin{equation}\label{e:slopes-Tosin}
\lim_{\xi\to\varpi^-}\varphi^{\, \prime}(\xi) =
 \left\{
 \begin{array}{ll}
 \lambda<0 & \text{ if \D{0} holds, }
 \\
 -\infty & \text{ if \D{1} or \D{2} hold, }
 \end{array}
 \right.
 \end{equation}
for some real number $\lambda$. 
\end{lemma} 
\begin{proof}
We integrate \eqref{e:CD} in $[\xi_0,\xi]\subset I$ and find
\begin{equation}\label{e:cd}
D\left(\phi(\xi)\right)\phi^{\prime}(\xi) -D\left(\phi(\xi_0)\right)\phi^{\prime}(\xi_0) + c\left[\phi(\xi)-\phi(\xi_0) \right] - f\left(\phi(\xi)\right) + f\left(\phi(\xi_0)\right)=0.
\end{equation}
By (D) and the behavior of $\phi$ at $-\infty$ we deduce that $\phi^{\prime}$ has finite limit at $-\infty$;  the boundedness of $\phi$ implies that $\phi^{\prime}(\xi_0) \to 0$ as $\xi_0 \to -\infty$. Therefore, if we pass to the limit in \eqref{e:cd} for $\xi_0 \to -\infty$, we obtain
\begin{equation}\label{e:123}
D\left(\phi(\xi)\right)\phi^{\prime}(\xi) =\left(\orho-\phi(\xi) \right)\left( c-\frac{-f(\phi(\xi)}{\orho-\phi(\xi)}\right), \quad \xi \in I.
\end{equation}
Denote $I=(-\infty, \xi_1)$ for some $\xi_1$. The case $\xi_1=\infty$ is excluded since every global traveling wave solution of equation \eqref{e:model_Tosin} is easily seen to be increasing by Proposition \ref{p:wfTosinI}; hence $\xi_1 $ is a real value. Moreover, when considering  condition (D), the strict concavity of $f$ and the estimate $c\le h(\orho)$, we can infer  from \eqref{e:123} that $\phi^{\, \prime}<0$ in $I$, so that $\lim_{\xi \to \xi_1^-} \phi(\xi)$ exists and it is necessarily $0$; it implies that $\xi_1=\varpi$. Finally, again from \eqref{e:123}, we obtain that
\begin{equation}\label{e:lim0}
\lim_{\xi \to \varpi^-}D\left(\phi(\xi)\right)\phi^{\prime}(\xi)=\lim_{\xi \to \varpi^-} c\left[\orho- \phi(\xi) \right]+f\left(\phi(\xi)\right)= c\orho\le f^{\prime}(\orho)\orho<0
\end{equation}
and claim \eqref{e:slopes-Tosin} follows from \eqref{e:lim0}.
\end{proof}

We notice that if $D$ is given by \eqref{e:v1}, then \D{2} holds and $\phi'(\xi)\to-\infty$ as $\xi\to\varpi^-$ by Lemma \ref{l:Tosin1}.

Second, we discuss the problem of the global existence (in the weak sense) of semi-wavefronts. This issue is crucial for the case $g\ne0$; the corresponding discussion is postponed to Section \ref{sec:pasting}. By Proposition \ref{p:wfTosinI}, equation \eqref{e:model_Tosin} admits classical wavefront solutions, which are always increasing; their presence makes the case of semi-wavefronts from $\orho$ different from that of semi-wavefronts to $\orho$. 

\begin{lemma}
If $c\notin \left(h(\orho,0)\right]$, then no strict classical semi-wavefront profile from or to $\orho$ can be extended to $\R$ by $0$ as a weak solution of \eqref{e:CD}.

If $c\in\left(h(\orho),0\right]$, then wavefront profiles exists. In this case, if $c\ne0$, they are classical and strictly monotone; if $c=0$, then the wavefront profile is weak, $\phi(\xi)=0$ if $\xi\in(-\infty,\xi_0]$ and $\phi'(\xi_0^+)=\infty$ for some $\xi_0\in\R$. 
\end{lemma}

\begin{proof}
We deal separately with the cases of semi-wavefronts from and to $\orho$. Recall Theorem \ref{t:swfwg}.

\begin{enumerate}[{\em (i)}]
 \item Let $\phi$ be a classical semi-wavefront profile {\em from} $ \orho$, with wave speed $c$, in the half-line $(-\infty, \varpi)$; then we have $c\le h( \orho)$. We claim that the extension $\tilde \phi$ of $\phi$ with $0$ to $[\varpi, \infty)$ is not a global (weak) solution of \eqref{e:CD}. Indeed, clearly $\tilde \phi$ is a solution in $\R\setminus\{\varpi\}$; by taking $I=\mathbb{R}$ in \eqref{d:tws} and a test function $\psi$ with $\psi(\varpi)\ne 0$, it is easy to see that $\tilde\phi$ is a weak solution of \eqref{e:CD} if and only if
$\lim_{\xi \to \varpi^-}D\left(\phi(\xi)\right)\phi^{\prime}(\xi)=0$.
This condition is never satisfied because of \eqref{e:lim0}. This proves our claim.
 
\item Let $\phi$ be a classical semi-wavefront profile {\em to} $\orho$; hence $c> h(\orho)=f^{\prime}(\orho)$. 
  
  If $c \in \left(f^{\prime}(\orho), 0\right]$, then there is a unique $\rho^- \in [0, \orho)$ such that
  $ c=\frac{-f(\rho^-)}{\orho -\rho^-}$.
 By Proposition \ref{p:wfTosinI}, equation \eqref{e:model_Tosin} has a wavefront solution with speed $c$; by uniqueness, the profile $\phi$ coincides with that wavefront and hence it is a global traveling-wave solution. 
     
If $c> 0$, then $\phi$ is a strict semi-wavefront solution to $\orho$, hence it is defined in some half-line $(\varpi, \infty)$; moreover, arguing as before we have that
$\lim_{\xi \to \varpi^+}D\left(\phi(\xi)\right)\phi^{\prime}(\xi)=c\orho>0$,
so, again, $\phi$ is not extendable in a weak sense by $0$.
 \end{enumerate}
\end{proof}
%A brief interpretation of these results now follows. The semi-wavefronts, whose existence is established in item {\em (i)} above, can be extended to $(\varpi,\infty)$ or $(-\infty,\varpi)$ by $0$, respectively, since $\rho=0$ is a solution to \eqref{e:Ewg}; we still denote by $\phi$ such semi-wavefronts. As a consequence, they give rise to {\em weak} global traveling-wave solutions, see Definition \ref{d:tws}, since differentiability is lost at $\xi=\varpi$; we notice that an analogous extension would not be possible for analogous solutions to \eqref{e:E} because $g(0)\ne0$ by assumption (g). Then, any initial data $\rho(x,0) = \phi_1(x)$ (resp., $\phi_2(x)$), see Figure \ref{f:SWFs}, gives rise to a semi-wavefront solution to \eqref{e:Ewg} that propagates with speed $c< h(\overline{\rho})$ (resp., $c \ge h(\overline{\rho})$); if \eqref{e:v1} holds, the situation represented in Figure \ref{f:SWFs} by $\phi_3$ does not occur. These solutions represent a crowd moving toward or leaving, respectively, an empty zone.

%We emphasize the influence of the diffusion on the patterns of solutions. Indeed, in the case $\delta=0$ and then $D=0$, equation \eqref{e:Ewg} reduces to a scalar conservation law with strictly decreasing propagation speed $h(\rho)$. As a consequence, an initial data $\rho(x,0) = \phi_1(x)$ gives a global smooth solution, which however does not conserve the shape of the initial datum; on the other hand, an initial data $\rho(x,0) = \phi_2(x)$ would lead to the formation a shock wave in finite time.

Again for the same model, we now consider the case when $g\ne0$; more precisely we focus on the case $g(\rho) = L\cdot(\orho-\rho)$. In the case $D$ is given by \eqref{e:v1}, condition \eqref{e:threshold} (see Theorem \ref{t:fosemi}) can be written as
\begin{equation}\label{e:interval}
\overline{v}\le c^*\le \overline{v} + v^*,
\end{equation}
where $v^*>0$ satisfies 
$
(v^*)^2 = 4L\delta \bar v\gamma\rho_0^{-2}(\orho-\rho_0)\, {\rm e}^{-\gamma\left(\frac{1}{\rho_0}-\frac{1}{\orho}\right)}$ for $\rho_0=\gamma/2+\orho - \sqrt{\left(\gamma/2\right)^2+\orho^2}$.

Even if semi-wavefront solutions $\rho(x,t)$ are important in several applications \cite{GK}, in the framework of collective movements their interest is limited by the fact that they are only defined in half-planes $x-ct\gtrless\varpi$. However, while referring to Section \ref{sec:pasting} for a discussion of the non-existence of global traveling-wave solutions, we provide here a simple application. Consider the initial-boundary value problem
\begin{equation}\label{e:IBVP}
\left\{
\begin{array}{ll}
\rho_t+\left(\rho v(\rho)\right)_x = \left(D(\rho)\rho_x\right)_x + g(\rho),& x<0, t>0,
\\
\rho(0,t)=\rho_b(t)&t>0,
\\
\rho(x,0)= \rho_0(x)&x<0,
\end{array}
\right.
\end{equation}
with $0\le \rho_0(x), \rho_b(t)\le\orho$ for every $x<0$ and $t>0$. Problem \eqref{e:IBVP} models a pedestrian motion in the half-line $x<0$, with initial datum $\rho_0$; pedestrians enter either through the $x$ axis with rate $g$ or through the boundary $x=0$ because of the term $\rho_b$. By Theorem \ref{t:semi}, we fix any $c>0$, denote by $\phi$ the correspondng semi-wavefront profile from $\orho$ and shift it so that it is defined in $I=(-\infty,\varpi]$ with $\varpi\ge0$. Then, we define $\rho(x,t)=\phi(x-ct)$, for $x<0$ and $t>0$; this definition makes sense because $c>0$. The function $\rho$ solves \eqref{e:IBVP} in the special case
$\rho_0(x):=\phi(x)$, $\rho_b(t)= \phi(-ct)$. In particular, according to Theorem \ref{t:strictly}, the road is completely filled in finite (or infinite) time depending on the source term $g$.

%A simple interpretation of Theorem \ref{t:2semi} now follows. Consider an initial datum
%\begin{equation}\label{e:initialdatum}
%\rho(x,0)=\left\{
%\begin{array}{ll}
%\phi_1(x)& \hbox{ if } x\le x_0,
%\\
%\phi_2(x)& \hbox{ if } x>x_0,
%\end{array}
%\right.
%\end{equation}
%for $x_0\in\R$, where $\phi_1$ and $\phi_2$ are provided by Theorem \ref{t:2semi} for $\xi_0=x_0$ and fixed $c$. This datum represents a zone of low density (around $\xi_0$) while the density is increasing both on the right and on the left of $x_0$ to the maximal value $\overline{\rho}$. In particular, if \eqref{e:g sub} holds, denote by $\overline{\xi}_1<\overline{\xi}_2$ the values corresponding to $\overline{\xi}$ in Theorem \ref{t:strictly} for $\phi_1$ and $\phi_2$, respectively; then the density is not maximal only in the space interval $(ct+\overline{\xi}_1,ct+\overline{\xi}_2)$. Theorem \ref{t:2semi} shows that such an initial datum \eqref{e:initialdatum} propagates with speed $c$ without changing its shape.

%%%%%%%%%%%%%%%%%%%%%%%%%%%%%%%%%%%%%%%%%%%%%%%%%%%%%%%%%%%%%%%%%%%%%%%%%%%%%%%%%%%%
\section{Comparison-type techniques}\label{s:comp}
\setcounter{equation}{0}

In this section we prove some results on the comparison-type techniques that we use in the following; we point out that the differentiability of $D$ is not required here. For $c\in\R$, $z_0 \ne 0$, $a\in [0, \overline \rho)$ and $b\in (0, \overline \rho]$, we introduce the following \emph{initial}- and \emph{final}-value problems corresponding to \eqref{e:zeq}:
\begin{equation}\label{e:vpt}
\left\{
\begin{array}{l}
\dot z (\varphi)= h(\varphi)-c-\frac{D(\varphi)g(\varphi)}{z(\varphi)},\ \phi>a,\\
z(a) = z_0,
\end{array}
\right.
\quad
\left\{
\begin{array}{l}
\dot z (\varphi)= h(\varphi)-c-\frac{D(\varphi)g(\varphi)}{z(\varphi)},\ \phi<b,\\
z(b)=z_0.
\end{array}
\right.
\end{equation}
In the following, we slightly simplify the limit notation used in the Introduction (see Theorem \ref{t:semiequiv}) for boundary values of solutions to singular equations as that in \eqref{e:vpt}: for instance, in the case $z_0=0$ we briefly write $z(a) = 0$ instead of $z(a^+) = 0$. Analogously, we use the notation $\dot{z}(0)$ for the right derivative of the function $z$ at $0$.

In Lemma \ref{l:iandfvp} we discuss the existence and uniqueness of solutions to both problems in \eqref{e:vpt} while in Lemma \ref{l:keyif} we show that the existence of a strict lower- or upper-solution for equation \eqref{e:zeq} determines an invariant region for the solutions of either $\eqref{e:vpt}_1$ or $\eqref{e:vpt}_2$.

\begin{lemma}\label{l:iandfvp} Consider the problems in \eqref{e:vpt}, for the above values of $c$, $z_0$, $a$ and $b$.
\begin{enumerate}[(1)]
\item The initial-value problem $\eqref{e:vpt}_1$ has a unique solution $z_a(\varphi)$ defined in its right maximal-existence interval $[a,\beta)$. In particular, $z_a(\beta)$ is a real value and $z_a(\beta)=0$ if $\beta< \overline \rho$.

\item The final-value problem $\eqref{e:vpt}_2$ has a unique solution $z_b(\varphi)$ defined on all $(0, b]$.
\end{enumerate}
\end{lemma}

\begin{proof} Denote by $f_c$ the right-hand side of equation \eqref{e:zeq}; since $f_c$ is globally continuous in its domain and locally Lipschitz-continuous in $z$, the uniqueness of the solutions of both $\eqref{e:vpt}_1$ and $\eqref{e:vpt}_2$ is guaranteed and it only remains to investigate their maximal-existence intervals.

The solutions of equation \eqref{e:zeq} never vanish in their domain because \eqref{e:zeq} is singular when $z=0$. Moreover, if $z(\varphi)$ is a positive solution in some interval $(a,b)\subseteq (0, \overline \rho)$, then the negative function $\eta(\varphi)=-z(\varphi)$ is a solution in $(a,b)$ of
\begin{equation}\label{e:eta}
\dot \eta (\varphi)= -h(\varphi)+c-\frac{D(\varphi)g(\varphi)}{\eta(\varphi)}.
\end{equation}
Since \eqref{e:eta} and \eqref{e:zeq} are completely analogous, we may restrict to the case $z_0<0$, see Figure \ref{f:ivp-fvp}. At last, assume that $z(\phi)$ is defined in some maximal-existence interval $(\alpha,\beta)$. Since $z(\varphi)<0$ in $(\alpha, \beta)$, the sign conditions in (D) and (g) imply that
\begin{equation}\label{e:bound}
\dot z(\varphi)>h(\varphi)-c,\qquad \hbox{ in $(\alpha, \beta)$}.
\end{equation}
Then, the function $z$ is bounded in $(\alpha,\beta)$. Now, we prove that both $z(\alpha)$ and $z(\beta)$ exist. Indeed, by multiplying by $z$ equation \eqref{e:zeq} we obtain that
\begin{equation*}
\frac12\frac{dz^2(\varphi)}{d\varphi} = \left(h(\varphi)-c\right)z(\varphi)-D(\varphi)g(\varphi).
\end{equation*}
By integrating in $[\varphi_0, \varphi]\subset (\alpha, \beta)$ we have
\begin{equation*}
z^2(\varphi_0)=z^2(\varphi)-2\int_{\varphi_0}^{\varphi}(h(\sigma)-c)z(\sigma)\, d\sigma +2\int_{\varphi_0}^{\varphi}D(\sigma)g(\sigma)\, d\sigma.
\end{equation*}
Since, moreover, $z(\varphi)<0$ in $(\alpha, \beta)$, we deduce
\begin{equation*}
z(\varphi_0) = -\sqrt{z^2(\varphi)-2\int_{\varphi_0}^{\varphi}(h(\sigma)-c)z(\sigma)\, d\sigma +2\int_{\varphi_0}^{\varphi}D(\sigma)g(\sigma)\, d\sigma},
\end{equation*}
which implies the existence of $z(\alpha)$. The existence of $z(\beta)$ is proved analogously.

%%%%%%%%%%%%%%%%%%%%%%%% Figure ivp-fvp
\begin{figure}[htbp]
\begin{picture}(100,100)(80,-15)
\setlength{\unitlength}{1pt}

\put(200,60){
\put(-15,0){\vector(1,0){175}}
\put(160,8){\makebox(0,0){$\varphi$}}
\put(-10,-5){\vector(0,1){30}}
\put(-5,17){\makebox(0,0){$z$}}
\put(10,2){\line(0,-1){80}}
\put(10,8){\makebox(0,0){$a$}}
\put(150,2){\line(0,-1){80}}
\put(150,8){\makebox(0,0){$b$}}

\put(0,0){\thicklines{\qbezier(10,-70)(80,-35)(100,0)}} %LR at 0
\put(60,-41){\thicklines{\vector(1,1){3}}}

\put(0,0){\thicklines{\qbezier(10,-70)(80,-45)(150,0)}} %LR at 0
\put(120,-19){\thicklines{\vector(2,1){3}}}

\put(0,0){\thicklines{\qbezier(10,-70)(100,-60)(150,-35)}} %LR
\put(88,-57){\thicklines{\vector(3,1){3}}}

\put(0,0){\thicklines{\qbezier(10,0)(100,-60)(150,-70)}} %RL at 0
\put(83,-44){\thicklines{\vector(-2,1){3}}}

\put(0,0){\thicklines{\qbezier(10,-40)(80,-70)(150,-70)}} %RL at 0
\put(58,-57){\thicklines{\vector(-2,1){3}}}

\put(-2,-70){\makebox(0,0)[l]{$z_0$}}
\put(164,-70){\makebox(0,0)[r]{$z_0$}}
}
\end{picture}
\caption{\label{f:ivp-fvp}{Solutions to the initial-value problem $\eqref{e:vpt}_1$ (left-to-right arrows) and to the final-value problem $\eqref{e:vpt}_2$ (right-to-left arrows); here, $z_0<0$.}}
\end{figure}
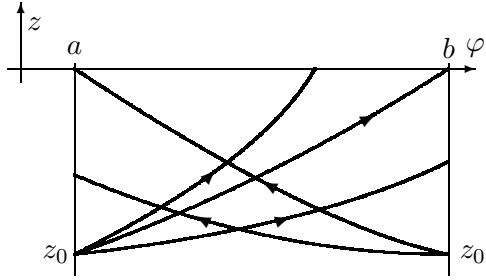
%%%%%%%%%%%%%%%%%%%%%%%%%%%%% End Figure ivp-fvp

\begin{enumerate}[{\em (1)}]

\item We showed above that $z_a(\beta)$ exists in $\R$; if $\beta <\overline \rho$, the continuation theorem for solutions of an ordinary differential equation implies $z_a(\beta)=0$.

\item Let $(\alpha, b] \subseteq (0, b]$ be the left maximal-existence interval of $z_b$ and assume by contradiction that $\alpha>0$. Since $z_b(\alpha)$ is a real value, then $z_b(\alpha)=0$ and so $z_b$ is continuously extendable to $[\alpha, b]$. Consider now a sequence $\{\psi_n\}_n \subset (\alpha, b]$ that converges to $\alpha$. By the mean value Theorem we find a sequence $\{\varphi_n\}_n\subset(a,b)$ with $\phi_n\in(\alpha,\psi_n)$ such that
\begin{equation}\label{e:mean}
\frac{z(\psi_n)}{\psi_n -\alpha}= \dot z(\varphi_n)<0
\end{equation}
for all $n \in \N$. If $\alpha>0$, from \eqref{e:zeq} we obtain that
 \begin{equation*}
\lim_{\varphi \to \alpha^+}\dot z(\varphi)=\lim_{\varphi \to \alpha^+} \left(h(\varphi)-c -\frac{D(\varphi)g(\varphi)}{z(\varphi)}\right) = +\infty,
\end{equation*}
in contradiction with \eqref{e:mean}. Hence, $\alpha=0$ and $z_b$ is defined on all $(0, b]$.
\end{enumerate}
\end{proof}
According to Lemma \ref{l:iandfvp}, every solution $z$ of \eqref{e:zeq} defined in $(0, b)\subseteq (0, \overline \rho]$ has a {\em continuous} extension to $[0, b)$, still denoted by $z$.

Now, we briefly recall the definitions of upper- and lower-solution for equation \eqref{e:zeq}.

\begin{definition}\label{d:ul} Let $J\subseteq [0,\overline \rho]$ be an interval. A function $\omega\in C^1(J)$ is a \emph{lower-solution} for equation \eqref{e:zeq} if
\begin{equation}\label{e:lower}
\dot\omega(\varphi)\le h(\varphi)-c -\frac{D(\varphi)g(\varphi)}{\omega (\varphi)},\qquad  \varphi\in J.
\end{equation}
Similarly, a function $\eta\in C^1(J)$ is an \emph{upper-solution} for \eqref{e:zeq} if
\begin{equation}\label{e:upper}
\dot\eta(\varphi)\ge h(\varphi)-c-\frac{D(\varphi)g(\varphi)}{\eta (\varphi)},\qquad  \varphi\in J.
\end{equation}
The function $\omega$ (resp. $\eta$) is a {\em strict} lower-solution (resp., upper-solution) if \eqref{e:lower} (resp., \eqref{e:upper}) holds with strict inequality.
\end{definition}

Now, we focus on the case $z_0<0$ in \eqref{e:vpt} and keep in mind Lemma \ref{l:iandfvp}. The existence of a strict lower- or upper-solution for equation \eqref{e:zeq} in either $[a, \beta)\subseteq [0, \overline \rho)$ or $(0,b]\subseteq (0, \overline \rho]$ determines an invariant region for the solutions of the \emph{initial}- and \emph{final}-value problems in \eqref{e:vpt}, respectively.
\begin{lemma}\label{l:keyif} Let $I\subseteq [0,\orho]$; consider a strict lower-solution $\omega$ and a strict upper-solution $\eta$ in $I$ of \eqref{e:zeq}, with $\omega(\phi)<0$ and $\eta(\phi)<0$ in $I$. Moreover, fix $z_0<0$.
\begin{enumerate}[(1)]
\item If $I=[a,b)$ and $z$ is the solution of $\eqref{e:vpt}_1$ defined in its maximal-existence interval $[a,\beta)\subseteq[a,b)$, then:
\begin{enumerate}[({1}.i)]

\item if $\omega(a)\le z_0$, then $\omega (\varphi)<z(\varphi)$ for all $\varphi \in (a, \beta)$;

\item if $\eta(a)\ge z_0$, then $\beta= b$ and $z(\varphi)<\eta(\varphi)$ for all $\varphi \in (a, b)$.
\end{enumerate}
\item If $I=(0,b]$ and $z$ is the solution of $\eqref{e:vpt}_2$, then:
\begin{enumerate}[({2}.i)]

\item if $\omega(b) \ge z_0$, then $\omega (\varphi)> z(\phi)$ for all $\varphi \in (0, b)$;

\item if $\eta(b)\le z_0$, then $\eta (\varphi)<z(\varphi)$ for all $\varphi \in (0, b)$.
\end{enumerate}
\end{enumerate}

\end{lemma}

%%%%%%%%%%%%%%%%%%%%%%%% Figure lower-upper

\begin{figure}[htbp]
\begin{picture}(100,80)(80,-10)
\setlength{\unitlength}{1pt}

\put(160,0){
%(a)
\put(-50,60){
\put(-15,0){\vector(1,0){165}}
\put(150,8){\makebox(0,0){$\varphi$}}
\put(-10,-5){\vector(0,1){20}}
\put(-5,17){\makebox(0,0){$z$}}
\put(10,2){\line(0,-1){80}}
\put(10,8){\makebox(0,0){$a$}}
\put(140,2){\line(0,-1){80}}
\put(140,8){\makebox(0,0){$b$}}
\put(0,0){\thicklines{\qbezier(10,-10)(50,-25)(140,-30)}} %eta
\put(70,-24){\thicklines{\vector(4,-1){3}}}
\put(0,-5){\thicklines{\qbezier(10,-20)(50,-40)(140,-45)}} %z
\put(70,-43){\thicklines{\vector(4,-1){3}}}
\put(0,-10){\thicklines{\qbezier(10,-30)(50,-50)(140,-60)}} %omega
\put(70,-60){\thicklines{\vector(4,-1){3}}}
\put(-2,-10){\makebox(0,0){$\eta(a)$}}
\put(0,-25){\makebox(0,0){$z_0$}}
\put(-2,-42){\makebox(0,0){$\omega(a)$}}
\put(105,-22){\makebox(0,0){$\eta$}}
\put(105,-41){\makebox(0,0){$z$}}
\put(105,-59){\makebox(0,0){$\omega$}}
}

%(b)
\put(170,60){
\put(-15,0){\vector(1,0){165}}
\put(150,8){\makebox(0,0){$\varphi$}}
\put(-10,-5){\vector(0,1){20}}
\put(-5,17){\makebox(0,0){$z$}}
\put(10,2){\line(0,-1){80}}
\put(10,8){\makebox(0,0){$a$}}
\put(140,2){\line(0,-1){80}}
\put(140,8){\makebox(0,0){$b$}}
\put(0,0){\thicklines{\qbezier(10,-10)(50,-25)(140,-30)}} %eta
\put(70,-24){\thicklines{\vector(-4,1){3}}}
\put(0,-5){\thicklines{\qbezier(10,-20)(50,-40)(140,-45)}} %z
\put(70,-43){\thicklines{\vector(-4,1){3}}}
\put(0,-10){\thicklines{\qbezier(10,-30)(50,-50)(140,-60)}} %omega
\put(70,-59){\thicklines{\vector(-4,1){3}}}
\put(143,-30){\makebox(0,0)[l]{$\omega(b)$}}
\put(143,-50){\makebox(0,0)[l]{$z_0$}}
\put(143,-70){\makebox(0,0)[l]{$\eta(b)$}}
\put(105,-22){\makebox(0,0){$\omega$}}
\put(105,-41){\makebox(0,0){$z$}}
\put(105,-59){\makebox(0,0){$\eta$}}
}
}
\end{picture}
\caption{\label{f:lower-upper}{Lower- and upper-solutions of $\eqref{e:vpt}_1$ (left) and $\eqref{e:vpt}_2$ (right).}}
\end{figure}
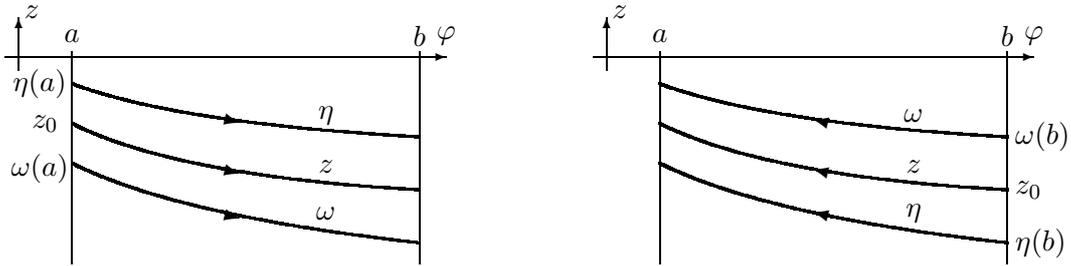
%%%%%%%%%%%%%%%%%%%%%%%%%%%%% End Figure lower-upper

\begin{proof} For both problems in $\eqref{e:vpt}$ we only prove case {\em (i)} since {\em (ii)} is similar; see Figure \ref{f:lower-upper}.

First, we deal with $\eqref{e:vpt}_2$. We claim that for some $\varepsilon >0$ we have $z(\varphi)<\omega(\varphi)$ for $\varphi\in (b-\varepsilon, b)$. Indeed, this follows by a continuity argument if $\omega(b)>z_0$; if $\omega(b)=z_0$, then $\dot \omega (b)<\dot z(b)$, because $\omega(\varphi)$ is a strict lower-solution. This proves the claim.

Now, assume that there exists $\varphi_0 \in (0, b)$ such that $z(\varphi_0)=\omega(\varphi_0)$; without loss of generality we can assume
\begin{equation}\label{e:te}
    z(\varphi)<\omega(\varphi), \qquad \varphi \in (\varphi_0, b).
     \end{equation}
As above, we obtain again that $\dot \omega(\varphi_0)<\dot z(\varphi_0)$ and then $z(\varphi) > \omega(\varphi)$ in a right neighborhood of $\varphi_0$,  in contradiction with \eqref{e:te}.

Now, we deal with $\eqref{e:vpt}_1$. If $\omega(0)<z_0$, then $\omega(\varphi)<z(\varphi)$ in a right neighborhood of $0$ by continuity. We reach the same conclusion if $\omega(0)=z_0$; indeed, $\omega$ is a strict lower-solution and then $\dot\omega(\phi)<\dot z(\phi)$ in a right neighborhood of $0$. Assume that there exists $\varphi_0  \in (0,b)$ in the domain of $z$ such that $\omega(\varphi_0) = z(\varphi_0)$; then, we easily get a contradiction as above.
\end{proof}

\section{The first-order problem}\label{s:first solution}
\setcounter{equation}{0}

In this section we first prove Theorem \ref{t:fosemi}. Then, we show some properties of the solutions of problem \eqref{e:fo}.

\smallskip

\begin{proofof}{Theorem \ref{t:fosemi}} We first deal with cases \D{1} and \D{2}, leaving \D{0} for the end of the proof. The existence of $c^*$ and the case $z(0)=0$ were considered in \cite[Theorem 2.2]{MMconv} under the further assumption $g(0)=0$. Indeed, the same result straightforwardly extends to cases \D{1} and \D{2} because $D(0)=0$. This proves the second part of the statement of the theorem. So, as far as existence is concerned, it remains to consider the case
\begin{equation}\label{e:cc*}
c<c^*
\end{equation}
and then $z(0)<0$. The proof splits into three parts, the last one dealing with uniqueness for $c\in\R$.

\medskip
\noindent  \emph{(a) Non-existence for large negative $z(0)$.} This first part does not assume \eqref{e:cc*}. We prove that if $z$ is a solution to \eqref{e:fo}, then necessarily $z(0)$ must satisfy the lower bound
\begin{equation}\label{e:lowerbound}
z(0) \ge -1-\overline \rho(H+M),
\end{equation}
for
\begin{equation}\label{e:HM}
H:=\max_{\varphi \in [0, \overline \rho]}h(\varphi)-c, \qquad M:=\max_{\varphi \in [0, \overline \rho]}D(\varphi)g(\varphi).
\end{equation}
Indeed, fix $z_0$ such that
\begin{equation}\label{e:z0}
z_0 < -1-\overline \rho(H+M)
\end{equation}
and consider the function 
\begin{equation}\label{e:etaline}
\eta(\varphi)=-\frac{1+z_0}{\overline \rho}\varphi+z_0,
\end{equation} 
i.e. the line connecting $(0, z_0)$ to $(\overline \rho, -1)$; see Figure \ref{f:manyzeta}. We claim that $\eta(\phi)$ is a strict upper-solution for \eqref{e:zeq} on all $[0, \overline \rho]$. Indeed, since $\eta(\varphi)\le -1$ for $\varphi \in [0, \overline \rho]$, we have that
$$
\frac{D(\varphi)g(\varphi)}{-\eta(\varphi)}=\frac{D(\varphi)g(\varphi)}{\frac{1+z_0}{\overline \rho}\varphi-z_0}\le D(\varphi)g(\varphi)\le M, \quad \varphi \in [0, \overline \rho].
$$
Consequently, by \eqref{e:z0} we have
$$
\dot \eta(\varphi)=-\frac{1+z_0}{\overline \rho}>H+M\ge h(\varphi)-c-\frac{D(\varphi)g(\varphi)}{\eta(\varphi)}, \quad \varphi \in [0, \overline \rho],
$$
which proves the claim.

Denote by $\hat z_c$ the solution of the equation in \eqref{e:fo} satisfying $\hat{z}_c(0) = z_0$, where $z_0$ satisfies \eqref{e:z0}. By Lemma \ref{l:iandfvp}{\em (1)} we have that $\hat z_c$ is unique; by Lemma \ref{l:keyif}{\em (1.ii)} that $\hat z_c$ is defined in $[0,\overline{\rho}]$ and $\hat z_c(\varphi)< \eta (\varphi)$ for all $\varphi \in (0, \overline \rho)$. Then, $\hat z_c(\overline \rho)\le \eta(\overline \rho)=-1$ and, hence, $\hat z_c$ is not a solution of \eqref{e:fo}.

\medskip

\noindent \emph{(b) Existence in cases \D{1} and \D{2}.} \ We denote by $z_{c^*}(\varphi)$ the solution of \eqref{e:fo} corresponding to $c^*$; the existence of $z_{c^*}(\phi)$ is guaranteed by the second part of the statement of the theorem and in particular $z_{c^*}(\phi)<0$ if $\phi\in(0,\orho)$. We also denote with $z_n(\varphi)$ the solution of the problem
\begin{equation}\label{e:zn}
\left\{
\begin{array}{l}
\dot z (\varphi)= h(\varphi)-c-\frac{D(\varphi)g(\varphi)}{z(\varphi)}, \quad \varphi \in (0, \overline \rho],\\
z(\overline \rho)=-\frac 1n,
\end{array}
\right.
\end{equation}
for $n\in \mathbb{N}$, which exists by Lemma \ref{l:iandfvp}{\em (2)}. By Lemma \ref{l:keyif}{\em (2.i)} we have
\begin{equation}\label{e:znz*}
z_n(\varphi)<z_{c^*}(\varphi), \quad \phi\in(0, \overline \rho]
\end{equation}
and then
\begin{equation}\label{e:zn<0}
z_n(\phi)<0, \quad \phi\in(0,\orho].
\end{equation}

%%%%%%%%%%%%%%%%%%%%%%%% Figure manyzeta
\begin{figure}[htbp]
\begin{picture}(100,120)(80,-30)
\setlength{\unitlength}{1pt}

\put(200,60){
\put(0,0){\vector(1,0){160}}
\put(160,8){\makebox(0,0){$\varphi$}}
\put(10,-80){\vector(0,1){100}}
\put(5,17){\makebox(0,0){$z$}}
\put(150,2){\line(0,-1){90}}
\put(150,8){\makebox(0,0){$\overline{\rho}$}}
\put(-5,-42){\makebox(0,0)[l]{$-1$}}
\put(8,-42){\line(1,0){4}}
\put(152,-42){\makebox(0,0)[l]{$-1$}}
\put(148,-42){\line(1,0){4}}
\put(0,-50){\thicklines{\qbezier(10,-20)(50,-35)(150,-35)}} %\hat{z}_c
\put(105,-75){\makebox(0,0){$\hat{z}_c$}}
\put(-2,-70){\makebox(0,0)[l]{$z_0$}}
\put(0,0){\thicklines{\qbezier(10,-38)(50,-35)(150,-10)}} %z_n
\put(105,-30){\makebox(0,0){$z_n$}}
\put(0,0){\thicklines{\qbezier(10,-28)(80,-25)(150,0)}} %\overline{z}
\put(25,-21){\makebox(0,0){$\overline{z}$}}
\put(0,-27){\makebox(0,0)[l]{$\ell$}}
\put(8,-28){\line(1,0){4}}
\put(154,-10){\makebox(0,0)[l]{$-\frac1n$}}
\put(148,-10){\line(1,0){4}}
\put(0,0){\thicklines{\qbezier(10,0)(70,-20)(150,0)}} %z_{c^*}
\put(20,-10){\makebox(0,0){$z_{c^*}$}}
\put(10,-70){\thicklines{\line(5,1){140}}}
\put(105,-46){\makebox(0,0){$\eta$}}
}
\end{picture}
\caption{\label{f:manyzeta}{The solutions $z_{c^*}$, $z_n$, $\overline{z}$, $\hat{z}_c$ and the upper-solution $\eta$; here, $z_0$ satisfies \eqref{e:z0}.}}
\end{figure}
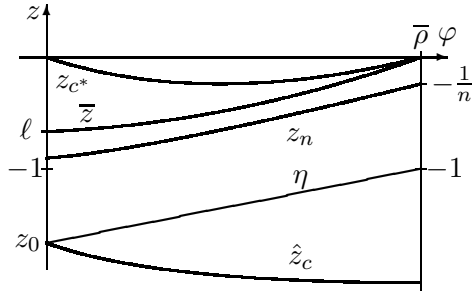
%%%%%%%%%%%%%%%%%%%%%%%%%%%%% End Figure manyzeta

Let $\hat z_c(\varphi)$ be the solution of the equation in \eqref{e:fo} with $\hat z_c(0) <-1-\overline \rho(H+M)$ that was already introduced in item \emph{(a)}. The uniqueness of solutions stated in Lemma \ref{l:iandfvp} implies, on the one hand, that $\hat z_c(\varphi)<z_n(\varphi)$ for all $\varphi \in [0, \overline \rho]$ and $n \in \N$; on the other hand, that the sequence $\{z_n\}_n$ is  increasing on $ (0, \overline \rho]$. Define
$$
\overline z(\varphi):=\lim_{n \to\infty}z_n(\varphi), \quad \varphi \in  (0, \overline \rho].
$$
By \eqref{e:znz*} we notice that
\begin{equation}\label{e:ozeta<0}
\overline z (\orho)=0 \quad \hbox{ and }\quad \overline{z}(\phi)<0,\ \phi\in(0,\orho).
\end{equation}
We claim that $\overline{z}$ is the solution to \eqref{e:fo} we are looking for. Indeed, by integrating the equation in \eqref{e:fo} in $[\varphi, \varphi_1]\subset (0, \overline \rho)$, we obtain that
\begin{equation}\label{e:Izn}
 z_n(\varphi_1)- z_n(\varphi)=\int_{\varphi}^{\varphi_1}\left (h(\sigma)-c\right) \, d\sigma +\int_{\varphi}^{\varphi_1}\frac{D(\sigma)g(\sigma)}{-z_n(\sigma)}\, d\sigma.
\end{equation}
Since the sequence
$$
\left\{\frac{D(\varphi)g(\varphi)}{-z_n(\varphi)} \right\}_n
$$
is positive by \eqref{e:zn<0} and increasing in $(0, \overline \rho)$, we can pass to the limit in \eqref{e:Izn} by the monotone convergence Theorem and obtain
\begin{equation}\label{e:Iz}
\overline z(\varphi_1)-\overline z(\varphi)=\int_{\varphi}^{\varphi_1}\left (h(\sigma)-c\right)\, d\sigma-\int_{\varphi}^{\varphi_1}\frac{D(\sigma)g(\sigma)}{\overline z(\sigma)}\, d\sigma.
\end{equation}
This implies that $\overline z(\varphi)$ is a solution of the equation in \eqref{e:fo} on all $(0, \overline \rho)$; it also satisfies $\overline z(\overline \rho)=0$ and $\overline z(\varphi)<0$ on $(0, \overline \rho)$. By \eqref{e:ozeta<0}, the function
$$
\varphi \longmapsto \int_{\varphi}^{\phi_1}\frac{D(\sigma)g(\sigma)}{\overline z(\sigma)}\, d\sigma, \quad \varphi \in (0, \overline \rho),
$$
is increasing. Then, identity \eqref{e:Iz} implies the existence of ${\lim_{\varphi \to 0^+}}\overline z(\varphi)=:\ell$; since $\overline{z}(\phi)<0$ if $\phi\in(0,\overline{\rho})$, we deduce that $\ell \in \{-\infty\}\cup (-\infty, 0]$. The case $\ell =0$ is excluded by the second part of the statement of the theorem because of \eqref{e:cc*}; moreover, we have $\hat z_c(\varphi)<z_1(\varphi)\le \overline z(\varphi)$ for all $\varphi \in (0, \overline \rho)$ and then $\ell$ is finite. In conclusion, we have $\ell \in (-\infty, 0)$.

\medskip
\noindent \emph{(c) Uniqueness in cases \D{1} and \D{2}.} Let $c\in \R$ and assume, by contradiction, that problem \eqref{e:fo} has two distinct solutions $z_1$ and $z_2$.

If \eqref{e:cc*} holds, we have $z_i(0)<0$, $i=1,2$, by the second part of the statement of the theorem, and $z_1(0) \ne z_2(0)$ by the unique solvability of $\eqref{e:vpt}_1$. We may assume that $z_1(0)<z_2(0)$, which yields $z_1(\varphi)<z_2(\varphi)<0$ for all $\varphi \in [0, \overline \rho)$.  Therefore
$$
\dot z_2(\varphi)-\dot z_1(\varphi) =\frac{D(\varphi)g(\varphi)}{-z_2(\varphi)}-\frac{D(\varphi)g(\varphi)}{- z_1(\varphi)}>0, \quad \mbox{for all }\varphi \in [0, \overline \rho),
$$
and then the function $z_2-z_1$ is increasing in $[0,\overline{\rho})$. As a consequence,
$$
\lim_{\varphi \to \overline \rho^{\, -}} \left( z_2(\varphi)- z_1(\varphi) \right) \ge z_2(0)-z_1(0)>0,
$$
in contradiction with $z_1(\overline \rho^{\, -})=z_2(\overline \rho^{\, -})=0$.  Hence, the uniqueness if proved if $c < c^*$.

If $c\ge c^*$, let $z_2$ be the solution satisfying $z_2(0^+)=0$ and $z_1$ another solution. By the uniqueness contained in the second part of the statement of the theorem, we have $z_1(0)<0$. Then, the arguments of the previous case apply and uniqueness is proved also in this case.

\medskip
\noindent \emph{(d) Existence and uniqueness in case \D{0}.} Now, we are left with case \D{0}. Let $\hat h$ be the even extension of $h$ to $[-\orho,0)$ and extend $g$ to the same interval with a continuous function $\hat g$ satisfying $\hat g(\rho)>0$ if $\rho\in[-\orho,0)$. We extend $D$ to $[-\orho,\orho]$ by a function $\hat D \in C^1[-\overline \rho , \overline \rho ]$ such that
\[
\hat D(-\overline \rho)=\dot {\hat D}(-\overline \rho)=0,\quad \hat D(\rho)>0,\quad \rho \in (-\overline \rho, 0).
\]
Then, instead of \eqref{e:fo} we consider the auxiliary  problem
\begin{equation}\label{e:fo 1}
\left\{
\begin{array}{l}
\dot z = \hat h(\varphi)-c-\frac{\hat D(\varphi)\hat g(\varphi)}{z}, \\
z(\varphi)<0,  \quad \varphi \in (-\overline \rho,\overline \rho),\\
z(-\overline \rho\,^+)=:z_0 \le 0, \quad z(\overline{\rho})=0.
\end{array}
\right.
\end{equation}
Problem \eqref{e:fo 1} has a unique solution $\hat z$ for all $c\in\R$; this follows by applying items {\em (b)} and {\em (c)} in the interval $[-\overline \rho,\overline \rho]$. It is easy to show that the restriction $z$ of $\hat z$ to $[0, \overline \rho]$ is a solution of problem \eqref{e:fo} with $z(0^+)<0$. This shows that also problem \eqref{e:fo} is uniquely solvable for all $c \in \R$. In conclusion, problem \eqref{e:fo} is uniquely solvable for all $c \in \R$ also under condition \D{0}.
\end{proofof}

Now, we prove the monotonicity with respect to $c$ of solutions to problem \eqref{e:fo}.

\begin{lemma}\label{l:slope z} \ Let $z_1$ and $z_2$ be solutions of problem \eqref{e:fo} corresponding to $c_1$ and $c_2$, respectively. If $c_1<c_2$, then we have that
\begin{equation}\label{e:ger}
z_1(\varphi)<z_2(\varphi), \qquad \varphi \in (0, \overline \rho).
\end{equation}
\end{lemma}

\begin{proof}
Since $c_1<c_2$, then $z_1$ is a strict upper-solution on $(0, \overline \rho)$ of equation \eqref{e:zeq} with $c=c_2$.
If there exists $\varphi_0 \in (0, \overline \rho)$ such that $z_2(\varphi_0) \le z_1(\varphi_0)$, then by Lemma \ref{l:keyif}{\em (1.ii)} we deduce that $z_2(\varphi)<z_1(\varphi)$ for $\phi\in(\varphi_0, \overline \rho)$. Hence,
\begin{equation*}
\dot z_2(\varphi)=h(\varphi)-c_2+\frac{D(\varphi)g(\varphi)}{-z_2(\varphi)}<
h(\varphi)-c_1+\frac{D(\varphi)g(\varphi)}{-z_1(\varphi)}=\dot z_1(\varphi), \quad \varphi \in (\varphi_0, \overline \rho),
\end{equation*}
which contradicts $z_2(\orho\,{}^-)=0$.
\end{proof}

We conclude this section with a result about the derivative $\dot z_c(0)$ of the solutions $z_c$ to \eqref{e:fo}, under conditions \D{1} or \D{2} and in the case $c\ge c^*$. Indeed, in the case $c < c^*$ or when \D{0} holds, we have $z(0)<0$ by Theorem \ref{t:fosemi}; then $z\in C^1[0,1)$ and $\dot{z}(0) = h(0)-c$ by \eqref{e:zeq}.

The existence of the slope $\dot z_c(0)$ was first proved in  \cite[Lemma 2.1]{MMconv} and the values of $\dot z_c(0)$ were obtained in \cite[Theorem 1.1]{MMconv}. However, since in \cite{MMconv} the assumption $g(0)=0$ holds, those computations can cover only our case \D{2}. To the best of our knowledge, the result of the following proposition in case \D{1} is new. Moreover, the proof of Proposition \ref{p:slopePD} unifies both cases \D{1} and \D{2}; we emphasize that it is completely different and simpler than that in \cite{MMconv} for the latter case.

\begin{proposition}\label{p:slopePD} Assume either {\rm \D{1}} or {\rm \D{2}} and let $z_c$ be the solution to problem \eqref{e:fo} for $c \ge c^*$. Then, $\dot z_c(0^+)$ exists and
\begin{equation}\label{e:cPDrr}
\dot z_c(0^+)  =
\left\{
\begin{array}{ll}
r_+(c) & \hbox{ if } c>c^*,
\\
r_-(c) & \hbox{ if } c=c^*.
\end{array}
\right.
\end{equation}
In particular, under assumption {\rm \D{2}} we have
\[
\dot z_c(0^+) = \left\{
\begin{array}{ll}
0 & \hbox{ if } c>c^*,
\\
h(0)-c^* & \hbox{ if } c=c^*.
\end{array}
\right.
\]
 \end{proposition}
\begin{proof} Let $c \ge c^*$ and assume, by contradiction, that $\dot z_c(0^+)$ does not exist. We notice that neither $\dot z_c(0^+)=+\infty$ nor $\dot z_c(0^+)=-\infty$ are possible, the latter because of \eqref{e:bound}. By Theorem \ref{t:fosemi}, we know that $z_c(0^+)=z_{c^*}(0^+)=0$; hence, there exist $-\infty\le l < L \le 0$ such that
\begin{equation}\label{e:Ll}
l=:\liminf_{\varphi \to 0^+}\frac{z_c(\varphi)}{\varphi} < \limsup_{\varphi \to 0^+}\frac{z_c(\varphi)}{\varphi}=:L\le 0.
\end{equation}
Let $\gamma \in (l, L)$ and consider a sequence $\{\sigma_n\}_n\subset (0, \overline \rho)$ such that $\sigma_n \to 0$, $\frac{z_c(\sigma_n)}{\sigma_n}=\gamma$ for every $n$ and also
\begin{equation*}
\frac{d}{d\varphi}\left( \frac{z_c(\varphi)}{\varphi}\right)_{\bigl\vert\varphi=\sigma_n}\ge 0.
\end{equation*}
Since
\begin{equation}\label{e:OO}
\frac{d}{d\varphi}\left( \frac{z_c(\varphi)}{\varphi}\right)=\frac{1}{\varphi}\left (\dot z_c(\varphi)-\frac{z_c(\varphi)}{\varphi} \right),
\end{equation}
we have
\begin{equation*}
\gamma \le \dot z_c(\sigma_n)=h(\sigma_n)-c-\frac{D(\sigma_n)g(\sigma_n)}{z_c(\sigma_n)}=
h(\sigma_n)-c-\frac{D(\sigma_n)g(\sigma_n)}{\gamma\sigma_n}, \qquad n \in \N.
\end{equation*}
When $n \to \infty$ we obtain
\begin{equation}\label{e:gamma1}
h(0)-c-\frac{\dot D(0)g(0)}{\gamma}\ge \gamma.
\end{equation}
In a similar way we can take a sequence $\{\delta_n\}_n\subset (0, \overline \rho)$ satisfying $\delta_n \to 0$, $ \frac{z_c(\delta_n)}{\delta_n}=\gamma$ for every $n$ and
\begin{equation*}
\frac{d}{d\varphi}\left( \frac{z_c(\varphi)}{\varphi}\right)_{\bigl\vert\varphi=\delta_n}\le 0.
\end{equation*}
Since  $\dot z_c(\delta_n)\le \gamma$ for all $n$ by \eqref{e:OO}, we obtain
\begin{equation*}
\gamma\ge \dot z_c(\delta_n)=h(\delta_n)-c-\frac{D(\delta_n)g(\delta_n)}{\gamma \delta_n}, \qquad n \in \N.
\end{equation*}
Then, by passing to the limit,
\begin{equation}\label{e:gamma2}
h(0)-c-\frac{\dot D(0)g(0)}{\gamma}\le \gamma.
\end{equation}
When combining \eqref{e:gamma1} and \eqref{e:gamma2} we obtain that $\gamma$ is a root of the second-order equation
$\gamma^2-(h(0)-c)\gamma+\dot D(0)g(0)=0$. This is in contradiction with \eqref{e:Ll} or because of the arbitrariness of $\gamma$; hence, $\dot z_c(0^+)$ exists for every $c\ge c^*$ and satisfies
\begin{equation*}
\dot z_c(0^+)\in \left\{ r_-(c), r_+(c)\right\}.
\end{equation*}
We remark that, according to \eqref{e:cs}, the r.h.s. in the previous formula is always defined in $\R$. Now, we notice that the function $\psi \colon [c^*, +\infty) \to \R$ defined by
\begin{equation*}
\psi(c) =\frac{h(0)-c- \sqrt{(h(0)-c)^2-4\dot D(0)g(0)}}{2}
\end{equation*}
is strictly decreasing. So, if we assume that $\dot z_c(0^+)=r_-(c)$ for some $c>c^*$, we obtain that $\dot z_{c^*}(0^+) > \dot z_c(0^+)$ both in the case $\dot z_{c^*}(0^+) = r_-(c^*)$ and $\dot z_{c^*}(0^+) = r_+(c^*)$. It implies that $z_{c^*}>z_c$ in a right neighborhood of $0$ in contradiction with Lemma \ref{l:slope z}. Formula \eqref{e:cPDrr} is then proved if $c>c^*$.

\smallskip

Now, assume $c=c^*$ and denote for short $r_-^* = r_-(c^*)$, $r_+^* = r_+(c^*)$.
%\begin{align*}
%r_-^* & =:\frac{h(0)-c^*-\sqrt{(h(0)-c^*)^2-4\dot D(0)g(0)}}{2}
%\\
%r_+^* &=:\frac{h(0)-c^*+\sqrt{(h(0)-c^*)^2-4\dot D(0)g(0)}}{2}.
%\end{align*}
If $r_-^* = r_+^*$, by \eqref{e:cs} we have  that $c^*=h(0)+2\sqrt{\dot D(0)g(0)}$ and estimate \eqref{e:cPDrr} is satisfied. It remains to consider the case $r_-^*<r_+^*$; again by \eqref{e:cs} we have
\begin{equation}\label{e:c^*greater}
c^*>h(0)+2\sqrt{\dot D(0)g(0)}.
\end{equation}
Let $\Delta>0$ be a positive value  satisfying $0<\Delta<r_+^*-r_-^*$.
We have $r_-^*\left(r_-^*+\Delta\right) > r_-^*\cdot r_+^* = \dot D(0) g(0)$
and hence
\begin{equation}\label{e:K}
-\frac{\dot D(0)g(0)}{r_-^*+\Delta}+\frac{\dot D(0)g(0)}{r_-^*}<\Delta.
\end{equation}
%Notice, in fact, that $r_-^*+\Delta<r_+^*<0$ and then the estimate \eqref{e:K}
%becomes
%\begin{equation}\label{e:perK}
%\dot D(0)g(0)<r_-^*(r_-^*+\Delta).
%\end{equation}
%Since it is easy to see that $r_-^* \cdot r_+^*=\dot D(0)g(0)$, by \eqref{e:perK}
%we have that
%\begin{equation*}
%\Delta<r_+^*-r_-^*
%\end{equation*}
%and this proves the claim \eqref{e:K}.
By \eqref{e:c^*greater} we can consider an increasing sequence $\{c_n\}_n\subset \bigl(h(0)+2\sqrt{\dot D(0)g(0)}, c^*\bigr)$ such that $c_n\to c^*$ as $n \to \infty$; let $\{z_n\}_n$ be the corresponding sequence of solutions to problem \eqref{e:fo} obtained in Theorem \ref{t:fosemi}. Notice, in particular, that $c_n<c^*$ implies
\begin{equation}\label{e:zn0}
z_n(0)<0, \qquad \text{for all }n \in\N.
\end{equation}
By Lemma \ref{l:slope z} we have that $\{z_n(\varphi)\}_n$ is an increasing sequence, for all $\varphi \in (0, \overline \rho)$, and $z_n(\varphi)<z_{c^*}(\varphi)$ in $(0, \overline \rho)$ for all $n \in \N$. As in the proof of item {\em (b)} in Theorem \ref{t:fosemi}, it is also possible to show that
\begin{equation}\label{e:lim z*}
\lim_{n \to \infty}z_n(\varphi)=z_{c^*}(\varphi), \qquad \varphi \in [0, \overline \rho].
\end{equation}
Because of \eqref{e:K}, we can introduce a positive value $\alpha$ such that
\begin{equation}\label{e:daKalpha}
-\frac{\dot D(0)g(0)}{r_-^*+\Delta}+\frac{\dot D(0)g(0)}{r_-^*}+\alpha<\Delta.
\end{equation}
By the continuity of the function $k(\varphi, c)=h(\phi) -c$ we can find $\sigma_ 1 >0$ and $\overline n \in \N $ such that
\begin{equation}\label{e:est3}
h(\varphi)-c_n<h(0)-c^*+\frac{\alpha}{2}, \qquad \text{for } \varphi \in (0, \sigma_1) \text{ and } n \ge \overline n.
\end{equation}
Moreover, conditions (g) and either \D{1} or \D{2} allow to determine a value $\sigma_2>0$ such that
\begin{equation}\label{e:est4}
-\frac{g(\varphi)}{r_-^*+\Delta} \cdot \frac{D(\varphi)}{\varphi}<-\frac{\dot D(0) g(0)}{r_-^*+\Delta}+\frac{\alpha}{2}, \qquad \varphi\in (0, \sigma_2).
\end{equation}
Denote $\sigma:=\min\{\sigma_1, \sigma_2\}$ and introduce the function $\eta \colon [0, \sigma] \to \R$ defined by  $\eta(\varphi):=(r_-^*+\Delta)\varphi$.
By \eqref{e:est3} and \eqref{e:est4} we have, for $n \ge \overline n$ and $\phi\in(0,\sigma)$,
\begin{equation*}
\begin{array}{rl}
h(\varphi)-c_n-\frac{D(\varphi)g(\varphi)}{\eta(\varphi)}=& h(\varphi)-c_n-\frac{g(\varphi)}{r_-^*+\Delta} \cdot \frac{D(\varphi)}{\varphi}\\
< &h(0)-c^*+\frac{\alpha}{2}-\frac{\dot D(0)g(0)}{r_-^*+\Delta}+\frac{\alpha}{2}\\
=&h(0)-c^*-\frac{\dot D(0)g(0)}{r_-^*}+\frac{\dot D(0)g(0)}{r_-^*}-\frac{\dot D(0)g(0)}{r_-^*+\Delta}+\alpha.
\end{array}
\end{equation*}
Notice that $h(0)-c^*-\frac{\dot D(0)g(0)}{r_-^*}=h(0)-c^*-r_+^*=r_-^*$. Hence, by \eqref{e:daKalpha} we obtain
\begin{equation*}
h(\varphi)-c_n-\frac{D(\varphi)g(\varphi)}{\eta(\varphi)} < r_-^*+\Delta=\dot \eta(\phi),
\end{equation*}
which shows that $\eta$ is a strict upper-solution of the equation in \eqref{e:fo} with $c=c_n$, $n \ge \overline n$, on all $(0, \sigma]$. Since $z_n(0)<0=\eta(0)$ by \eqref{e:zn0}, a continuity argument shows that there exists $\psi_n\in (0, \sigma)$ such that $z_n(\varphi)<\eta(\varphi)$ in $[0, \psi_n)$ for all $n \ge \overline n$. In the remaining interval $[\psi_n, \sigma]$ we can apply Lemma \ref{l:keyif}{\em (1.ii)}; in conclusion we obtain $z_n(\varphi)<\eta(\varphi)$ in $[0, \sigma)$. Then,
\begin{equation*}
\frac{z_n(\varphi)}{\varphi}<\frac{\eta(\varphi)}{\varphi}=r_-^*+\Delta, \qquad \varphi \in (0, \sigma), \, n \ge \overline n.
\end{equation*}
Finally, by estimate  \eqref{e:lim z*} we have that
\begin{equation*}
\frac{z_{c^*}(\varphi)}{\varphi} \le r_-^*+\Delta<r_+^*, \qquad \varphi \in (0, \sigma).
\end{equation*}
We deduce that $\dot z_{c^*}(0^+)<r_+^*$; then, condition \eqref{e:cPDrr} holds and the proof is complete.
\end{proof}

%%%%%%%%%%%%%%%%%%%%%%%%%%%%%%%%%%%%%%%%%%%%%%%%%%%%%%%%%%%%%%%%%%%%%%%%%%%%%%%%%%%%

\section{Semi-wavefronts  via a first-order analysis}\label{s:equiv}
\setcounter{equation}{0}

In this section we first show that semi-wavefronts of equation \eqref{e:E} are strictly monotone. Then, by exploiting this result, we prove Theorem \ref{t:semiequiv}.

Here follows our first result: we recall that by Definition \ref{d:swf} a semi-wavefront is necessarily valued in $[0,\overline{\rho})$.

\begin{proposition}\label{p:monot}  \
Let $\varphi$ be a semi-wavefront of \eqref{e:E} from (to) $\overline \rho$. Then $\varphi^{\, \prime}(\xi)< 0$ ($\varphi^{\, \prime}(\xi) > 0$, respectively) for all $\xi$ in the domain of $\phi$ such that $0< \varphi(\xi) < \overline \rho$.
\end{proposition}
\begin{proof}
We only consider the case of a semi-wavefront $\phi$ from $\overline{\rho}$; the other case is analogous.

Let $\varphi$ be defined on the half-line $(-\infty, \varpi)$, with $\varpi\in \R$;  we assume that there exists $\xi_0 \in (-\infty, \varpi)$ with $\varphi(\xi_0)\in (0, \overline \rho)$ such that $\varphi^{\, \prime}(\xi_0)=0$. We denote
\[
T(\xi):=D\left(\varphi(\xi)\right)\varphi^{\, \prime}(\xi), \quad  \xi \in (-\infty, \varpi).
\]
We have that  $T(\xi_0)=0$; by \eqref{e:tws}, condition (g) and the assumption $\phi(\xi_0)<\overline{\rho}$, we deduce $T^{\, \prime}(\xi_0)=-g(\varphi(\xi_0))<0$. Hence, we have that  $(\xi_0-\xi)T(\xi)>0$ for $\xi\ne \xi_0$ in a neighborhood of $\xi_0$. By condition (D), it follows that  $\xi_0$ is a local maximum point of $\varphi$. The boundary condition $\varphi(-\infty)=\overline \rho$ then implies that there exists a local minimum point $\xi_1<\xi_0$ of $\varphi$, in contradiction with the previous discussion. Hence $\varphi(\xi)>0$ for  $\xi \in (-\infty, \varpi)$ and $\varphi^{\, \prime}(\xi)<0$ whenever $0<\varphi(\xi) < \overline \rho$.
\end{proof}

\begin{remark}\label{rem:smoothphi}
{\rm Let $\varphi(\xi)$ be a semi-wavefront for \eqref{e:E} from (to) $\overline \rho$. Proposition \ref{p:monot} shows that there exists an interval $I \subseteq (-\infty, \varpi)$ (resp., $I \subseteq (\varpi, +\infty)$), such that $0<\varphi(\xi)<\overline \rho$ for $\xi \in I$. By arguing on the smoothness of the terms in \eqref{e:tws} it is not difficult to show that $\varphi \in C^2(I)$.
}\end{remark}

\begin{remark}\label{rem:inverse}{\rm Proposition \ref{p:monot} implies that every semi-wavefront $\phi(\xi)$ has inverse $\xi=\xi(\varphi)$ defined on  $[0,\overline \rho)$ and $\xi(0) = \varpi$. Moreover, if $\varphi(\xi)$ is a wave profile from $\overline \rho$ we have that either $\xi(\overline{\rho}^{\, -}) = \xi_0\in\R$ or, if $\phi$ is strictly monotonic, that $\xi(\overline{\rho}^{\, -}) = -\infty$; an analogous property holds if $\varphi(\xi)$ is a wave profile to $\overline \rho$.}\end{remark}

To prove Theorem  \ref{t:semiequiv} we need the following lemma, which concerns the asymptotic behavior of semi-wavefronts.

\begin{lemma}\label{l:ab} \ Let $\varphi $ be a semi-wavefront of \eqref{e:E} from $\overline \rho$ defined on the half-line $(-\infty,\varpi)$. Then
\begin{enumerate}[(i)]

\item $\varphi^{\, \prime}(\xi) \to 0$ as $\xi \to -\infty$;

\item $D\left(\varphi(\xi)\right)\varphi^{\, \prime}(\xi) \to \ell$ as $\xi \to \varpi^{\, -}$, for some real value $\ell\le 0$.
\end{enumerate}
\end{lemma}

\begin{proof} First, we prove {\em (i)}. By integrating \eqref{e:tws} in $[\xi_0, \xi] \subset (-\infty, \varpi)$ we obtain
\begin{align}
\lefteqn{D\left(\varphi(\xi_0)\right)\varphi^{\, \prime}(\xi_0)=}
\label{e:l0}
\\
&=D\left(\varphi(\xi)\right)\varphi^{\, \prime}(\xi)+c\left( \varphi(\xi)-\varphi(\xi_0)\right) - \mathcal{H}\left(\varphi(\xi)\right) +
  \mathcal{H}\left(\varphi(\xi_0)\right)+\ds\int_{\xi_0}^{\xi} g\left(\varphi(s)\right)\, ds,
\label{e:l}
\end{align}
where $\mathcal{H}(r):=\int_0^{r}h(s)\, ds$  for $r \in [0, \overline \rho]$.
If $\xi_0 \to -\infty$, then  $\mathcal{H}\left(\varphi(\xi_0)\right) \to \mathcal{H}(\overline \rho)$; in addition, according to (g), the limit
\begin{equation}\label{e:limit}
{\lim_{\xi_0 \to -\infty}}\int_{\xi_0}^{\xi}g\left(\varphi(s)\right)\, ds
\end{equation}
exists. Since \eqref{e:l0} is negative by Proposition \ref{p:monot}\emph{(i)}, the limit \eqref{e:limit} is surely a real value. Hence, we proved the existence of
\begin{equation*}
\lim_{\xi_0 \to -\infty}D\left(\varphi(\xi_0)\right)\varphi^{\, \prime}(\xi_0)=:\lambda\in \R.
\end{equation*}
This implies that $\lim_{\xi_0 \to -\infty}\varphi^{\, \prime}(\xi_0) = \lambda/D(\overline \rho)$ and, since $\varphi$ is bounded, we conclude that $\lambda =0$. This proves {\em (i)}.

Now, we prove {\em (ii)}. By \eqref{e:l0}-\eqref{e:l}, it is immediate to see that the limit of $D\left(\varphi(\xi)\right)\varphi^{\, \prime}(\xi)$ for $\xi\to\varpi^-$ exists and it is a value in $(-\infty, 0]$. The lemma is completely proved.
\end{proof}

An analogous result can be easily proved if $\phi$ is a semi-wavefront to $\orho$.
Now, we can prove Theorem \ref{t:semiequiv}.

\smallskip

\begin{proofof}{Theorem \ref{t:semiequiv}}
Let $\varphi$ be a semi-wavefront of \eqref{e:E} with speed $c\in \mathbb{R}$ from $\overline \rho$; by Remark \ref{rem:inverse} we denote by $\xi(\varphi)$ its inverse function, which is defined at least for $\varphi \in [0, \overline \rho)$. The function $z(\varphi)=D(\varphi)\varphi^{\, \prime}\left(\xi(\varphi)\right)$ clearly satisfies the first equation in \eqref{e:fo} for the same $c$; moreover, $z(\phi)<0$ for $\phi\in(0,\orho)$ by Proposition \ref{p:monot}, $z(\orho)=0$ by Lemma \ref{l:ab}{\em (i)} and $z(0^+)\le0$ by Lemma \ref{l:ab}{\em (ii)}. Therefore $z$ satisfies problem \eqref{e:fo}.

\smallskip Conversely, let $z(\varphi)$ be a solution of \eqref{e:fo} for some $c\in \mathbb{R} $ and $\varphi(\xi)$ the solution of the initial-value problem
\begin{equation}
\left\{
\begin{array}{l}
\varphi^{\, \prime}(\xi)=\frac{z(\varphi)}{D(\varphi)},\\
\varphi(0)=\frac{\overline \rho}{2},
\end{array}
\right.
\label{e:phirho2}
\end{equation}
in its maximal existence interval $(\alpha, \varpi)$; this means that $\phi$ satisfies
\begin{equation*}
\lim_{\xi \to \alpha^+} \varphi(\xi)=\overline \rho, \qquad \qquad \lim_{\xi \to \varpi^-} \varphi(\xi)=0.
\end{equation*}
If there exists $\hat \alpha \in (\alpha,\varpi)$ satisfying $\varphi(\hat \alpha)=\overline \rho$, by condition (D) and \eqref{e:phirho2} we deduce
\begin{equation*}
\lim_{\xi \to \hat \alpha^+}\varphi^{\, \prime}(\xi)=\lim_{\varphi \to \overline \rho^{\,-}}\frac{z(\varphi)}{D(\varphi)}=0.
\end{equation*}
Here, we used the assumption $D(\orho)>0$, which is contained in (D). Hence, we can continue $\varphi(\xi)$ to the left of $\hat \alpha$ with $\overline{\rho}$ in a differentiable way; see what we pointed out below the statement of Theorem \ref{t:semi}. Therefore we can assume that $\alpha=-\infty$ and then $\varphi(-\infty)=\overline \rho$.

To complete the proof we need to show that the semi-wavefront is strict, i.e., that $\varpi$ is finite; the proof depends on the values of $c$.

In case \D{0}, we always have $z(0^+)<0$ for all $c$; then, by \eqref{e:phirho2}, we obtain that
\begin{equation}\label{e:z(0)D0}
\lim_{\xi \to \varpi^{-}}\varphi^{\, \prime}(\xi)=\lim_{\varphi \to 0^+}\frac{z(\varphi)}{D(\varphi)}=\frac{z(0^+)}{D(0)}<0.
\end{equation}
In particular, we have that $\varpi\in\R$ and the slope of the semi-wavefront never equals $-\infty$. Then, we focus on cases \D{1} and \D{2}.

\smallskip

\noindent {\em (a)} $c<c^*$. In this case Theorem \ref{t:fosemi} implies $z(0)<0$; by \eqref{e:phirho2} we deduce that, in both cases \D{1} and \D{2},
\begin{equation}\label{e:inomega}
 \lim_{\xi \to \varpi^-}\varphi^{\, \prime}(\xi)=\lim_{\varphi \to 0^+}\frac{z(\varphi)}{D(\varphi)}=-\infty,
\end{equation}
 and hence $\varpi \in \mathbb{R}$.

\smallskip

\noindent {\em (b)} $c=c^*>h(0)$. In case \D{1}, by definition of derivative we deduce as above that
    \[
    \lim_{\xi\to\varpi^-}\phi'(\xi) = \frac{r_-{(c^*)}}{\dot D(0)}.
    \]
    In case \D{2} we have $z(0)=0$ and $\dot z(0)=h(0)-c^*<0$ by Theorem \ref{t:fosemi} and Proposition \ref{p:slopePD}, respectively. Since $\dot D(0)=0$, we deduce again \eqref{e:inomega} and hence $\varpi\in \R$ in both cases.

\smallskip

\noindent {\em (c)} $c>c^*$.
In case \D{1} we have
\[
z(0)=0\quad \hbox{ and } \quad \dot z(0)=r_+(c).
\]
Then,
\[
\lim_{\xi\to\varpi^-}\phi'(\xi) = \frac{r_+{(c)}}{\dot D(0)}
\]
and so $\varpi$ is finite. In case \D{2}, Theorem \ref{t:fosemi} and Proposition \ref{p:slopePD} imply
\begin{equation}\label{e:dotz0}
z(0)=0\quad \hbox{ and } \quad \dot z(0)=0.
\end{equation}
The situation is more delicate than in the previous cases, since  we need to construct suitable lower- and upper-solutions of \eqref{e:zeq} in a sharp way.

Fix $\varepsilon >0$ and denote
$$
\eta(\varphi):=-\frac{g(0)}{c -h(0)-\varepsilon g(0)}D(\varphi).
$$
Since $g(0)>0$, the function  $\eta(\varphi)$ is defined and negative on all $ (0, \overline \rho)$ for every sufficiently small $\varepsilon $.  Moreover, as $\phi\to0^+$ we have both $\dot \eta(\varphi)\to 0$, by $\dot D(0)$ in \D{2}, and
$$
h(\varphi)-c-\frac{D(\varphi)g(\varphi)}{\eta(\varphi)}=h(\varphi)-c+\frac{c -h(0)-\varepsilon g(0)}{g(0)}g(\varphi) \to -\varepsilon g(0)<0.
$$
Then, we can find  $\sigma \in(0, \overline \rho]$ such that
$$
\dot \eta(\varphi)> h(\varphi)-c-\frac{D(\varphi)g(\varphi)}{\eta(\varphi)}, \qquad \text{ for } \varphi \in (0, \sigma],
$$
i.e. $\eta(\varphi)$ is a strict upper-solution for \eqref{e:zeq} on $(0, \sigma]$.

By \eqref{e:dotz0} and the mean value Theorem, there is a sequence $\{\varphi_n\}_n \subset (0, \overline \rho)$, with $\varphi_n \to 0^+$, such that $\dot z(\varphi_n) \to 0$; this implies that
\begin{equation}\label{e:phin}
\frac{D(\varphi_n)g(\varphi_n)}{z(\varphi_n)} \to h(0)-c
\end{equation}
when $n \to \infty$.  Consequently we have
\begin{align*}
\displaystyle{\lim_{n \to \infty}} \frac{\eta(\varphi_n)}{z(\varphi_n)} & = \displaystyle{\lim_{n \to \infty}}-\frac{g(0)}{c -h(0)-\varepsilon g(0)}\frac{D(\varphi_n)}{z(\varphi_n)}
\\
& = -\frac{g(0)}{c -h(0)-\varepsilon g(0)}\left( - \frac{c-h(0)}{g(0)}\right)
 = \frac{c -h(0)}{c -h(0)-\varepsilon g(0)}>1.
\end{align*}
Hence, we can find $\hat{\sigma} \in (0, \sigma]$ such that $z(\hat \sigma)>\eta(\hat \sigma)$ and by Lemma \ref{l:keyif}{\em (2.ii)} we conclude that $z(\varphi)>\eta\emph{}(\varphi)$ on all $(0, \hat \sigma)$. Then
\begin{equation}\label{e:bd}
\frac{D(\varphi)}{z(\varphi)}< \frac{D(\varphi)}{\eta(\varphi)}=-\frac{c-h(0)}{g(0)}+\varepsilon, \qquad \varphi \in (0, \hat \sigma).
\end{equation}
We proceed in an analogous way with lower-solutions. Consider the function
$$
\omega(\varphi):=-\frac{g(0)}{c -h(0)+\varepsilon g(0)}D(\varphi),
$$
which is defined and negative on all $ (0, \overline \rho)$. For $\varphi \to 0^+$ we have that  $\dot \omega(\varphi)\to 0$ and
$$
h(\varphi)-c-\frac{D(\varphi)g(\varphi)}{\omega(\varphi)}=h(\varphi)-c+\frac{c -h(0)+\varepsilon g(0)}{g(0)}g(\varphi) \to \varepsilon g(0)>0.
$$
Then, we can find  $\mu \in(0, \overline \rho]$ such that
$$
\dot \omega(\varphi)< h(\varphi)-c-\frac{D(\varphi)g(\varphi)}{\omega(\varphi)}, \qquad \varphi \in (0, \mu],
$$
i.e. $\omega(\varphi)$  is a strict lower-solution for the equation in \eqref{e:fo} on $ (0, \mu]$.  Moreover, if $(\varphi_n)_n $ satisfies \eqref{e:phin}, we have that
\begin{align*}
\displaystyle{\lim_{n \to \infty}} \frac{\omega(\varphi_n)}{z(\varphi_n)} & = \displaystyle{\lim_{n \to \infty}} \left(-\frac{g(0)}{c -h(0)+\varepsilon g(0)}\frac{D(\varphi_n)}{z(\varphi_n)}\right)
\\
& = -\frac{g(0)}{c -h(0)+\varepsilon g(0)} \left( -\frac{c-h(0)}{g(0)}\right)= \frac{c -h(0)}{c -h(0)+\varepsilon g(0)}<1.
\end{align*}
Hence we can find $\hat{\mu} \in (0, \mu]$ such that $z(\hat \mu)<\omega(\hat \mu)$ and according to Lemma \ref{l:keyif}{\em (2.i)} we conclude that $z(\varphi)<\omega(\varphi)$ on all $(0, \hat \mu)$. Then
\begin{equation}\label{e:bs}
\frac{D(\varphi)}{z(\varphi)}>\frac{D(\varphi)}{\omega(\varphi)}=-\frac{c-h(0)}{g(0)}-\varepsilon, \qquad \varphi \in (0, \hat \mu).
\end{equation}
By combining \eqref{e:bd} with \eqref{e:bs}, and since $\varepsilon$ is arbitrary, we conclude that
\begin{equation}\label{e:slopec}
\lim_{\varphi \to 0^+}\frac{D(\varphi)}{z(\varphi)}=-\frac{c-h(0)}{g(0)}.
\end{equation}
We notice that the limit in \eqref{e:slopec} is nontrivial since both $D(0)=\dot D(0)=0$ by \D{2} and $z(0)=\dot z(0)=0$ by Theorem \ref{t:fosemi} and \eqref{e:dotz0}. Formula \eqref{e:slopec} implies that $\varpi$ is finite also in this case and that $\varphi^{\, \prime}(\varpi^{\, -})=-\frac{g(0)}{c-h(0)}$.

\smallskip

\noindent {\em (d)} $c=c^*=h(0)$. Because of \eqref{e:cs}, this case does not occur under \D{1} but only under \D{2}. For $\varepsilon >0$ we denote  $\omega(\varphi):=-\frac{D(\varphi)}{\varepsilon}$ for $\varphi \in (0, \overline \rho)$. Reasoning as in \emph{(c)}, it is possible to find $\mu \in (0, \overline \rho]$ such that $\omega(\varphi)$ is a strict lower-solution for the equation in \eqref{e:fo} on $(0, \mu]$. Moreover, by using the sequence $\{\varphi_n\}$ that we exploited to prove \eqref{e:phin}, we have
%\begin{equation*}
%\lim_{n \to \infty}\frac{\omega(\varphi_n)}{z(\varphi_n)}=0.
%\end{equation*}
%This implies $z(\hat \mu)<\omega(\hat \mu)$ for some $\hat \mu \in (0, \mu]$.
\begin{equation*}
0=\lim_{n \to \infty}\dot z(\varphi_n)=\lim_{n \to \infty} \left(h(\varphi_n)-c^*-\frac{D(\varphi_n)g(\varphi_n)}{z(\varphi_n)}\right).
\end{equation*}
Since $h$ is continuous with $h(0)=c^*$ and according to (g), we obtain
\begin{equation*}
\lim_{n \to \infty}\frac{D(\varphi_n)}{z(\varphi_n)}=0.
\end{equation*}
It implies
\begin{equation*}
z(\varphi_n)< -\frac{D(\varphi_n)}{\varepsilon}=\omega(\varphi_n),
\end{equation*}
for sufficiently large $n$. Since $\varphi_n \to 0^+$, it is then possible to find $\hat \mu \in (0, \mu]$ satisfying
$z(\hat \mu)<\omega(\hat \mu)$. By Lemma \ref{l:keyif}{\em (2.i)}, we conclude that $z(\varphi)<\omega(\varphi)$ for $\varphi \in (0, \hat \mu]$ and hence
$$
-\varepsilon<\frac{D(\varphi)}{\omega(\varphi)}< \frac{D(\varphi)}{z(\varphi)}<0,\quad \varphi \in (0, \hat \mu].
$$
Consequently, we have
\begin{equation}\label{e:ed}
\lim_{\varphi \to 0^+}\frac{D(\varphi)}{z(\varphi)}=0.
\end{equation}
Then, we have again $\varpi \in \R$ and $\phi'(\varpi^-)=-\infty$.
\end{proofof}

\section{Proof of the main results}\label{s:proofs}
\setcounter{equation}{0}

In this section we finally prove the theorems stated in Section \ref{s:main}.

\begin{proofof}{Theorem \ref{t:swfwg}}
Consider the new equation
\begin{equation}\label{e:Ewg0}
\rho_t + \bar f(\rho)_x=\left( \bar D(\rho)\rho_x\right)_x,
% + \left(H(\rho)\right)_x,
\qquad (x,t)\in\R\times[0,+\infty),
\end{equation}
for $\bar f(\rho) = f(\orho-\rho)-f(\orho)$ and $\bar D(\rho) = D(\orho-\rho)$.
%\begin{equation*}
%H(\rho)=:\int_0^{\rho}h(\overline{\rho} -s) \, ds=\int_{\overline{\rho}-\rho}^{\overline{\rho}}h(\sigma) \, d\sigma \qquad \rho \in [0, \overline \rho].
%\end{equation*}
Notice that
\begin{equation*}
\lim_{s \to 0^+}\frac{\bar f(s)}{s} = -\lim_{s \to 0^+}h(\orho-s) = -h(\orho)
%\lim_{s \to 0^+}\frac{-H(s)}{s}=\lim_{s \to 0^+}-h(\overline{\rho}-s)=-h(\overline{\rho})
\end{equation*}
and define $H(s)=-\bar f(s)  - h(\orho)s$, for $s\in [0, \overline \rho]$. So, we can apply \cite[Theorem 5.1]{GK} and conclude that equation \eqref{e:Ewg0} has exactly one semi-wavefront (say $\psi(\zeta)$, $\zeta \in (\omega, +\infty)$) decreasing to $0$ for $c > -h(\orho)$, exactly one such solution for $c= -h(\orho)$, provided that $H(s)>0$ for $s$ in a right neighborhood of $0$, and no such solutions for $c<-h(\orho)$. Since equation \eqref{e:Ewg0} can be equivalently written as $\rho_t - h(\orho-\rho)\rho_x=\left( D(\orho-\rho)\rho_x\right)_x$, it is clear (see equation \eqref{e:tws}) that the function $\psi(\zeta)$ is a solution of
\begin{equation*}
 \left( D(\orho-\psi)\psi^{\, \prime}\right)^{\, \prime}+\left(c+h(\orho-\psi)\right)\psi^{\, \prime}=0, \qquad {}^{\, \prime}=\frac{d}{d\zeta},
\end{equation*}
for $\zeta \in (\omega, +\infty)$. Let $\xi := -\zeta \in (-\infty,\varpi)$ with $\varpi:= -\omega$, and $\varphi(\xi)=:\overline{\rho} -\psi(\zeta)$; the function the  $\varphi(\xi)$  satisfies $\varphi(\xi) \to \overline{\rho}$ as $\xi \to -\infty$ and also
\begin{equation*}
 \left( D(\varphi)\varphi^{\, \prime}\right)^{\, \prime}+\left(-c-h(\varphi)\right)\varphi^{\, \prime}=0, \qquad \xi \in (-\infty,\varpi) \quad \text{with }{}^{\, \prime}=\frac{d}{d\xi}.
\end{equation*}
We obtained that $\varphi(\xi)$ is a semi-wavefront of \eqref{e:Ewg} from $\orho$ with wave speed $-c$ and also the converse is true, i.e. to every  semi-wavefront of \eqref{e:Ewg} there corresponds one of \eqref{e:Ewg0}. This proves the statements concerning the existence of semi-wavefronts from $\orho$ as well as their uniqueness up to shifts. The results about semi-wavefronts to $\orho$ are easily deduced arguing as above or by a change of variables as in the proof of Theorem \ref{t:semi}. At last, the smoothness property follows by Remark \ref{rem:smoothphi}.
\end{proofof}

\begin{proofof}{Theorem \ref{t:semi}} \ The existence and uniqueness (up to shifts) of classical semi-wavefronts {\em from} $\overline{\rho}$ is a direct consequence of Theorems \ref{t:fosemi} and \ref{t:semiequiv}.

\smallskip

Now, we show the existence of a unique (up to shifts) semi-wavefront {\em to} $\overline \rho$ for every $c\in\R$. Given $c \in \R$, consider the semi-wavefront from $\overline \rho$ of the equation
\begin{equation}\label{e:ER}
\rho_t-h(\rho)\rho_x=\left(D(\rho)\rho_x\right)_x+g(\rho), \qquad (x,t)\in\R\times[0,\infty),
\end{equation}
with speed $-c$ and profile $\psi(\xi)$ satisfying $\psi(0)=0$. As already remarked in the Introduction, the profile $\psi(\xi)$ is a solution of
\begin{equation}\label{e:R}
\left(D\left(\psi(\xi)\right)\psi^{\, \prime}(\xi)\right)^{\, \prime} + \left(-c+h\left(\psi(\xi)\right)\right)\psi^{\, \prime}(\xi) + g\left(\psi(\xi)\right)=0, \qquad  \xi \in (-\infty, 0).
\end{equation}
We define $\varphi(\xi):=\psi (-\xi)$ for $\xi \in (0, +\infty)$. We notice that $\varphi^{\, \prime}(0^{\, +})>0$
%(possibly $\varphi^{\, \prime}(0^{\, +})=+\infty$, see Remark \ref{rem:phi_smooth})
by \eqref{e:slope0}--\eqref{e:slope2} and that $\varphi(\xi) \to\overline \rho$ as $\xi \to +\infty$. Moreover, for $\xi \in (0, +\infty)$ and ${}'=d/d\xi$, we have that
\begin{equation*}
\begin{array}{rl}
\left(D\left(\varphi(\xi)\right)\varphi^{\, \prime}(\xi)\right)^{\, \prime}=&
-\left(D\left(\psi(-\xi)\right)\psi^{\, \prime}(-\xi)\right)^{\, \prime}
\\
=&-\left(-c+h\left(\psi(-\xi)\right)\right)\psi^{\, \prime}(-\xi) - g\left(\psi(-\xi)\right)\\
\\
=&-\left(c-h\left(\varphi(\xi)\right)\right)\varphi^{\, \prime}(\xi) - g\left(\varphi(\xi)\right).
\end{array}
\end{equation*}
Hence the function $\varphi$ satisfies \eqref{e:tws} on all $(0, +\infty)$ and then it is a semi-wavefront of \eqref{e:E} to $\overline \rho$.

About uniqueness, we argue conversely: starting from a semi-wavefront to $\overline \rho$ of \eqref{e:E} and reasoning as before, we obtain a semi-wavefront from $\overline \rho$ of \eqref{e:ER} with opposite speed and with $-h$ replacing $h$. Therefore, up to shifts, equation \eqref{e:E} has exactly one semi-wavefront to $\overline \rho$ for every wave speed.

The smoothness of the semi-wavefronts follows by Remark \ref{rem:smoothphi}; formulas \eqref{e:slope0}--\eqref{e:slope2} follow by \eqref{e:z(0)D0}, \eqref{e:inomega}, \eqref{e:slopec} and \eqref{e:ed}.

At last, we are left with the proof of \eqref{e:remark}. We claim that
\begin{equation}\label{e:ClaimLimit}
\lim_{\xi \to \varpi^-} \left(\varphi_1^{\, \prime}(\xi)-\varphi_2^{\, \prime} (\xi)\right)\in [-\infty, 0).
\end{equation}
Let us briefly show how \eqref{e:ClaimLimit} implies \eqref{e:remark}. Formula \eqref{e:ClaimLimit} implies $\varphi_1^{\, \prime}<\varphi_2^{\, \prime}$ in a left neighborhood  $I$ of $\varpi$; by applying the Mean Value Theorem to $\varphi_1-\varphi_2$, we get estimate \eqref{e:remark} in $I$.
Assume by contradiction that there exists $\overline \xi \in (-\infty, \varpi)$ satisfying  $\varphi_1(\overline \xi)=\varphi_2(\overline \xi)=:\overline \varphi \in (0, \orho)$; without loss of generality we can suppose
\begin{equation}\label{e:ss}
\varphi_2(\xi)<\varphi_1(\xi), \qquad \text{for  } \xi \in (\overline \xi, \varpi).
\end{equation}
By Lemma \ref{l:slope z} we get
$$
D(\overline \varphi)\varphi_1^{\, \prime}(\overline \xi) = D(\overline \varphi)\varphi_1^{\, \prime}\left(\xi_1(\overline \varphi)\right) = z_1(\overline \varphi)<z_2(\overline \varphi) = D(\overline \varphi)\varphi_2^{\, \prime}\left(\xi_2(\overline \varphi )\right)=D(\overline \varphi)\varphi_2^{\, \prime}(\overline \xi),
$$
where $\xi_1$ and $\xi_2$ denote the inverse functions of $\phi_1$, $\phi_2$, respectively, see Remark \ref{rem:inverse}. We deduce that $\varphi_1^{\, \prime}(\overline \xi)<\varphi_2^{\, \prime}(\overline \xi)$, which contradicts \eqref{e:ss}. This would prove \eqref{e:remark}.

The proof of \eqref{e:ClaimLimit} is split into four parts, see the proof of Theorem \ref{t:semiequiv}.

\begin{enumerate}[{\em (a)}]
\item \emph{Assume \D{0}}.  The definition of $z(\varphi)$ implies that
$$
\lim_{\xi \to \varpi^-} \left(\varphi_1^{\, \prime}(\xi)-\varphi_2^{\, \prime}(\xi)\right) = \lim_{\varphi \to 0^+}\frac{z_1(\varphi)-z_2(\varphi)}{D(\varphi)}=\frac{z_1(0)-z_2(0)}{D(0)}.
$$
By estimate \eqref{e:ger} and Theorem \ref{t:fosemi} we get $z_1(0)\le z_2(0)<0$. If $z_1(0)=z_2(0)$, then $\dot z_1(0)=h(0)-c_1 - \frac{D(0)g(0)}{z_1(0)}>h(0)-c_2 = \dot z_2(0) -\frac{D(0)g(0)}{z_2(0)}$, in contradiction with \eqref{e:ger}. Hence, $z_1(0)<z_2(0)$ and \eqref{e:ClaimLimit} holds.

\item \emph{Assume \D{1} or \D{2}, with $c_1<c^*$.} By Theorem \ref{t:fosemi} we have $z_1(0)<0$.

    If $c_2 \le c^*$, by arguing as in case \emph{(a)}, we conclude that
    $$
    \lim_{\xi \to \varpi^-} \left(\varphi_1^{\, \prime}(\xi)-\varphi_2^{\, \prime}(\xi)\right) = \lim_{\varphi \to 0^+}\frac{z_1(\varphi)-z_2(\varphi)}{D(\varphi)}=-\infty.
    $$

    If $c_2 > c^*$, then by \eqref{e:slope1} or \eqref{e:slope2} (in case \D{1} or \D{2}, respectively) we have that $\varphi_2^{\, \prime}(\xi)$ has a finite limit when $\xi \to \varpi^-$ and then
    \begin{equation*}
    \lim_{\xi \to \varpi^-}\left( \varphi_1^{\, \prime}(\xi)-\varphi_2^{\, \prime}(\xi)\right) = \lim_{\varphi \to 0^+} \frac{z_1(\varphi)}{D(\varphi)} - \lim_{\xi \to \varpi^-} \varphi_2^{\, \prime}(\xi)=-\infty.
    \end{equation*}

\item \emph{Assume \D{1}, with $c_1\ge c^*$.} By \eqref{e:r_pm}, recall that $r_-(c^*)<r_+(c^*)$ and also that $r_+(c)$ is increasing for $c\ge c^*$. Then, according to \eqref{e:slope1}, we get
    \begin{equation*}
    \lim_{\xi \to \varpi^-}\left(\varphi_1^{\, \prime}(\xi)-\varphi_2^{\, \prime}(\xi)\right) = \left\{ \begin{array}{rl} \frac{r_-(c^*)-r_+(c_2)}{\dot D(0)}<0, & \text{if } c_1=c^*,
    \\[2mm]
    \frac{r_+(c_1)-r_+(c_2)}{\dot D(0)}<0, & \text{if } c_1>c^*.\end{array}\right.
    \end{equation*}

\item \emph{Assume \D{2}, with $c_1\ge c^*$.} The estimate \eqref{e:cs} implies that  $c_1\ge h(0)$; hence, from \eqref{e:slope2} we have
    \begin{equation*}
    \lim_{\xi \to \varpi^-} \left(\varphi_1^{\, \prime}(\xi)-\varphi_2^{\, \prime}(\xi)\right) = \left\{
    \begin{array}{rl} -\infty, & \text{if } c_1=c^*,
    \\
    \frac{g(0)}{\left(c_1-h(0)\right)\left(c_2-h(0)\right)}\,(c_1-c_2)<0, & \text{if } c_1>c^*.\end{array}\right.
    \end{equation*}
\end{enumerate}
This completes the proof of \eqref{e:ClaimLimit} and then of Theorem \ref{t:semi}.
\end{proofof}

%%%%%%%%%

\begin{proofof}{Theorem \ref{t:strictly}} \ We prove the result only in the case of semi-wavefronts from $\overline \rho$; the same conclusions can be easily drawn for semi-wavefronts to $\overline \rho$ with the change of variables exploited in the proof of Theorem \ref{t:semi}. Moreover, we assume without any loss of generality that $\rho_1=0$ both in \eqref{e:pend g} and \eqref{e:g sub}: if $\rho_1>0$, it is sufficient either to increase $L$ in \eqref{e:pend g} or decrease it in \eqref{e:g sub} to a new constant $\overline{L}$ such that both \eqref{e:pend g} and \eqref{e:g sub} hold in $[0,\overline{\rho}]$ with $L$ replaced by $\overline{L}$.

%it is sufficient to localize the arguments below in the interval $[\rho_1,\orho]$.

Let $\varphi$ be a semi-wavefront in $(-\infty, \varpi)$ with speed $c$. Denote by $z(\varphi)$, $\varphi \in [0, \overline \rho]$, the solution of \eqref{e:fo} with the same wave speed $c$ provided by Theorem \ref{t:fosemi}. Let $\xi(\phi)$ be the inverse function of $\phi$, see Remark \ref{rem:inverse}, and denote
\begin{equation}\label{e:overlinexi}
\overline{\xi} := \lim_{\phi\to\orho\,^-}\xi(\phi).
\end{equation}

\smallskip

{\em Case (i).}\ For $n \in \N$ and $a>0$, we denote $\eta_n(\varphi):=a(\varphi-\overline \rho)-\frac 1n, \, \varphi\in [0, \overline \rho]$. First, we show that it is possible to find $a$, independently from $n$, such that $\eta_n$ is a strict upper-solution of \eqref{e:zeq} on $[0, \overline \rho)$ for all $n$. Indeed, by \eqref{e:pend g} (with $\rho_1=0$), we have that
\begin{equation}\label{e:hcD}
h(\varphi)-c-\frac{D(\varphi)g(\varphi)}{\eta_n(\varphi)}\le H + K\frac{L(\overline \rho -\varphi)}{a(\overline \rho - \varphi)+\frac 1n},
\end{equation}
where $H$ was defined in \eqref{e:HM} and $K:=\max_{\varphi \in [0, \overline \rho]}D(\varphi)$.
The function
\begin{equation*}
\varphi \longmapsto \frac{L(\overline \rho -\varphi)}{a(\overline \rho - \varphi)+\frac 1n}, \qquad \varphi\in [0, \overline \rho],
\end{equation*}
is strictly decreasing and then, by \eqref{e:hcD},
\begin{equation*}
h(\varphi)-c-\frac{D(\varphi)g(\varphi)}{\eta_n(\varphi)}\le H + \frac{K L\orho}{a\overline{\rho}+\frac 1n}<H+\frac{K L}{a}.
\end{equation*}
We have that $H+\frac{KL}{a}<a$ if we choose
\begin{equation}\label{e:a}
a>\frac{H+\sqrt{H^2+4KL}}{2}.
\end{equation}
With this choice, the function $\eta_n$ is a strict upper-solution of \eqref{e:zeq} in $[0, \overline \rho)$ for all $n$. This proves our claim.

Moreover, since  $z(\overline \rho^{\, -}) = 0 > -\frac 1n = \eta_n(\overline \rho)$, we can find $\hat \varphi_n \in (0, \overline \rho)$ satisfying $z(\varphi)>\eta_n(\varphi)$ for $\varphi\in [\hat \varphi_n, \overline \rho]$.
If we apply Lemma \ref{l:keyif}\emph{(2.ii)} in the remaining interval $[0, \hat \varphi_n)$ we conclude that
\begin{equation}\label{e:eta n}
z(\varphi)>\eta_n(\varphi), \qquad \text{for all }\varphi\in (0, \overline \rho] \text{ and } n\in\N.
\end{equation}
Since $D(\orho)>0$ by (D), then $\delta:= \min_{\phi\in[\orho/2,\orho]}D(\phi)>0$; as a consequence, by \eqref{e:overlinexi} we have that
\begin{align*}
\overline{\xi} - \xi\left(\frac{\overline \rho}{2}\right) &= \int_{\frac{\overline \rho}{2}}^{\overline \rho}\xi^{\, \prime}(\varphi)\, d\varphi
 = \int_{\frac{\overline \rho}{2}}^{\overline \rho}\frac{1}{\varphi^{\, \prime}(\xi(\varphi))}\, d\varphi=\int_{\frac{\overline \rho}{2}}^{\overline \rho}\frac{D(\varphi)}{z(\varphi)}\, d\varphi
\\
& < \int_{\frac{\overline \rho}{2}}^{\overline \rho}\frac{D(\varphi)}{\eta_n(\varphi)}\, d\varphi< \delta\int_{\frac{\overline \rho}{2}}^{\overline \rho}\frac{1}{a(\varphi-\overline \rho)-\frac 1n}\, d\varphi=\frac \delta a \ln \frac{2}{na\overline \rho  +2}.
\end{align*}
If we pass to the limit for $n \to \infty$ in the above lines, we see that the right-hand side tends to $-\infty$; hence, $\overline{\xi}=-\infty$.

\medskip

{\em Case (ii).}\ The proof is similar to that of Case {\em (i)} but the choice of a lower-solution (instead of an upper-solution) is more tricky. More precisely, we define $\overline h:=\min_{\varphi \in [\orho/2, \overline \rho]}\left(h(\varphi)-c\right)$, fix a value $\beta \in (\frac{\alpha+1}{2}, 1)$  and take $k>0$ satisfying
\begin{equation}\label{e:Deltak}
\frac{\delta L}{k}-k\beta\left( \frac{\overline\rho}{2}\right)^{2\beta -(\alpha+1)}>0,
\end{equation}
where $\delta$ is defined as in case {\em (i)}. For every $n \in \mathbb{N}$ with $\frac{\overline\rho}{2}<\overline \rho -\frac 1n$, we introduce the function $\omega_n(\varphi) \colon [\frac{\overline\rho}{2}, \overline \rho] \to \mathbb{R}$ defined by
  \begin{equation*}
  \omega_n(\varphi) =\left\{\begin{array}{ll}-k(\overline \rho -\frac 1n -\varphi)^{\beta} & \quad  \varphi\in  [\frac{\overline\rho}{2},
  \overline \rho -\frac 1n],\\[1mm]
  0 & \quad \varphi\in  (\overline \rho -\frac 1n,\orho].
  \end{array}
  \right.
  \end{equation*}
We claim that
\begin{equation}\label{e:claim}
\omega_n(\varphi)\ge z(\varphi), \qquad \varphi\in  [\overline\rho/2,
  \overline \rho].
\end{equation}
Indeed, since $z(\varphi)<0$ in the interval $(0, \overline \rho)$, by a continuity argument we can find $\psi_n \in (\frac{\overline \rho}{2}, \overline \rho -\frac 1n )$ such that $\omega_n(\varphi)\ge z(\varphi)$ on $[\psi_n, \overline \rho]$. If we show that $\omega_n$ is a strict lower-solution of \eqref{e:zeq} on $[\frac{\overline \rho}{2}, \psi_n]$, then we can apply Lemma \ref{l:keyif}\emph{(2.i)} in the interval $(\frac{\overline \rho}{2}, \psi_n]$ and prove \eqref{e:claim}.

According to \eqref{e:g sub} (with $\rho_1=0$), we obtain, for $\varphi \in [\frac{\overline \rho}{2}, \psi_n]$,
\begin{align}
h(\varphi)-c-\frac{D(\varphi)g(\varphi)}{\omega_n(\varphi)}
&=
h(\varphi)-c+\frac{D(\varphi)g(\varphi)}{k(\overline \rho -\frac 1n -\varphi)^{\beta}}\ge \overline h +\frac{\delta L(\overline \rho -\varphi)^{\alpha}}{k(\overline \rho -\frac 1n -\varphi)^{\beta}}\nonumber
\\
& =
\overline h +\frac{\delta L(\overline \rho -\varphi)^{\alpha}}{k(\overline \rho -\frac 1n -\varphi)^{\alpha}}\cdot \frac{1}{(\overline \rho -\frac 1n -\varphi)^{\beta-\alpha}}\nonumber
\\
& \ge
\overline h + \frac{\delta L}{k}\frac{1}{(\overline \rho -\frac 1n -\varphi)^{\beta-\alpha}}.\label{e:second}
\end{align}
%Notice that
%\begin{equation*}
%\left(\frac{\overline \rho -\varphi}{\overline \rho -\frac 1n -\varphi} \right)^{\alpha}\ge \left(\frac{\overline \rho -\frac{\overline\rho}{2}}{\overline \rho -\frac 1n -\frac{\overline\rho}{2}} \right)^{\alpha}>1, \quad \varphi \in  [\frac{\overline\rho}{2}, \overline \rho ].
%\end{equation*}
%Hence, from \eqref{e:first}, we have that
%\begin{equation}\label{e:second}
%h(\varphi)-c-\frac{D(\varphi)g(\varphi)}{\omega_n(\varphi)}\ge \overline h + \frac{\delta L}{k(\overline \rho -\frac 1n -\varphi)^{\beta-\alpha}}.
%\end{equation}
%
Now, we introduce the function $\eta_n \colon [\frac{\overline\rho}{2}, \psi_n ] \to \mathbb{R}$ defined by
\begin{equation*}
 \eta_n(\varphi)=\overline h + \frac{\delta L}{k}\frac{1}{(\overline \rho -\frac 1n -\varphi)^{\beta-\alpha}}-\dot \omega_n(\varphi)
\end{equation*}
and notice that $\orho-\frac1n -\phi<\frac{\orho}{2}$ for $\varphi \in [\frac{\overline \rho}{2}, \orho-\frac1n]$; we deduce, for $\varphi \in [\frac{\overline \rho}{2}, \psi_n]$,
\begin{align}
 \frac{\delta L}{k}- k\beta\left(\overline \rho -\frac 1n -\varphi\right)^{2\beta-(1+\alpha)}
 & > \frac{\delta L}{k}-k\beta\left( \frac{\overline \rho}{2} \right)^{2\beta-(1+\alpha)},
\label{e:third}
\\
 \frac{1}{(\overline \rho -\frac 1n -\varphi)^{\beta-\alpha}}
 &> \frac{1}{\left(\frac{\overline \rho}{2} \right)^{\beta-\alpha}}.
\label{e:fourth}
\end{align}
By means of the definition of $\omega_n$ and \eqref{e:third}, \eqref{e:fourth}, \eqref{e:Deltak}, we have that, for $\varphi \in [\frac{\overline \rho}{2}, \psi_n]$,
\begin{align*}
 \eta_n(\varphi) & = \overline h + \frac{\delta L}{k}\frac{1}{(\overline \rho -\frac 1n -\varphi)^{\beta-\alpha}}- \frac{k\beta}{(\overline \rho -\frac 1n -\varphi)^{1-\beta}}
 \\
 & = \overline h + \frac{1}{(\overline \rho -\frac 1n -\varphi)^{\beta-\alpha}}\left[ \frac{\delta L}{k}- k\beta\left(\overline \rho -\frac 1n -\varphi\right)^{2\beta-(1+\alpha)}\right]\\
 & > \overline h+\frac{1}{\left(\frac{\overline \rho}{2} \right)^{\beta-\alpha}}\left[\frac{\delta L}{k}-k\beta\left(\frac{\orho}{2}\right)^{2\beta-(1+\alpha)} \right]>0,
\end{align*}
if $k$ is sufficiently small. Hence,
 \begin{equation*}
\overline h + \frac{\delta L}{k}\frac{1}{(\overline \rho -\frac 1n -\varphi)^{\beta-\alpha}}>\dot \omega_n(\varphi), \quad \phi\in\left[\frac{\overline\rho}{2}, \psi_n\right].
 \end{equation*}
Then, by \eqref{e:second}, $\omega_n$ is a strict lower-solution of \eqref{e:zeq} on $(\frac{\overline\rho}{2}, \psi_n]$ and \eqref{e:claim} is proved.

The sequence $\{\omega_n\}_n$ is  monotone and
\begin{equation*}
\lim_{n \to \infty}\omega_n(\varphi)=-k(\overline \rho-\varphi)^{\beta}:=\omega(\varphi), \quad \varphi
\in [\overline\rho/2,
  \overline \rho].
\end{equation*}
By \eqref{e:claim} we have $\omega(\varphi)\ge z(\varphi)$ for $\varphi \in [\frac{\overline\rho}{2}, \overline \rho ]$ and, as in Case \emph{(i)}, by \eqref{e:overlinexi} we get
\begin{align*}
\overline \xi  - \xi(\frac{\overline \rho}{2}) &= \int_{\frac{\overline \rho}{2}}^{\overline \rho}\xi^{\, \prime}(\varphi)\, d\varphi
 = \int_{\frac{\overline \rho}{2}}^{\overline \rho}\frac{1}{\varphi^{\, \prime}(\xi(\varphi))}\, d\varphi=\int_{\frac{\overline \rho}{2}}^{\overline \rho}\frac{D(\varphi)}{z(\varphi)}\, d\varphi
\\
& \ge \int_{\frac{\overline \rho}{2}}^{\overline \rho}\frac{D(\varphi)}{\omega(\varphi)}\, d\varphi\ge  -\frac{K}{k}\int_{\frac{\overline \rho}{2}}^{\overline \rho}\frac{1}{(\overline \rho-\varphi)^{\beta}}\, d\varphi=-\frac{K}{k(1-\beta)}\left(\frac{\overline \rho}{2} \right)^{1-\beta},
\end{align*}
where $K$ was defined below \eqref{e:hcD}. Therefore $\overline \xi \in \mathbb{R}$ and so $\varphi(\xi)\equiv \overline \rho$ for $\xi \le \overline \xi$.
\end{proofof}

\section{On the existence of global traveling-wave solutions}\label{sec:pasting}
\setcounter{equation}{0}

As we noticed in Section \ref{s:example}, the existence of a semi-wavefront solution is a notable theoretical result but global solutions can be more interesting in some applications. The existence of semi-wavefront profiles for any $c\in\R$ is a motivation to the following construction. An analogous procedure is well known and fully characterized for some dispersive equations \cite{Lenells_CH, Lenells_Classification, Lenells_R}. 

We fix a wave speed $c$ and $\varpi\in\R$. Theorem \ref{t:semi}, together with a shift argument, provides us of a semi-wavefront solution $\rho_1$ from $\orho$ with wave profile $\phi_1$ and a semi-wavefront solution $\rho_2$ to $\orho$ with wave profile $\phi_2$, both of them with the same speed $c$ and satisfying $\phi_1(\varpi)=\phi_2(\varpi)=0$. Such wave profiles are unique by the same theorem. We define, see Figure \ref{f:pasting},
\begin{equation}\label{e:phi}
\phi(\xi)=\left\{
\begin{array}{rl}
\phi_1(\xi) & \hbox{ if } \xi\le \varpi,\\
\phi_2(\xi) & \hbox{ if } \xi> \varpi.
\end{array}
\right.
\end{equation}

%%%%%%%%%%%%%%%%%%%%%%%% Figure pasting

\begin{figure}[htbp]
\begin{picture}(100,100)(80,-10)
\setlength{\unitlength}{1pt}
%(a)
\put(150,0){
\put(40,0){\vector(1,0){200}}
\put(40,60){\line(1,0){200}}
\put(240,8){\makebox(0,0){$\xi$}}
\put(130,-10){\vector(0,1){100}}
\put(137,87){\makebox(0,0){$\phi$}}
\put(137,67){\makebox(0,0){$\overline{\rho}$}}

\put(0,0){\thicklines{\qbezier(40,58)(120,58)(160,0)}}
\put(60,45){\makebox(0,0)[b]{$\phi_1$}}

\put(0,0){\thicklines{\qbezier(160,0)(200,56)(240,58)}}
\put(220,43){\makebox(0,0)[b]{$\phi_2$}}

\put(162,30){\makebox(0,0)[b]{$\phi(\xi)$}}

%\multiput(159,0)(0,5){3}{$.$}
\put(160,-2){\makebox(0,0)[t]{$\varpi$}}

%\multiput(130,9)(5,0){6}{$.$}
%\put(128,10){\makebox(0,0)[r]{$\rho_0$}}
}

\end{picture}
\caption{\label{f:pasting}{Pasting two wave profiles $\phi_1$ and $\phi_2$ to get a global profile $\phi$.}}
\end{figure}
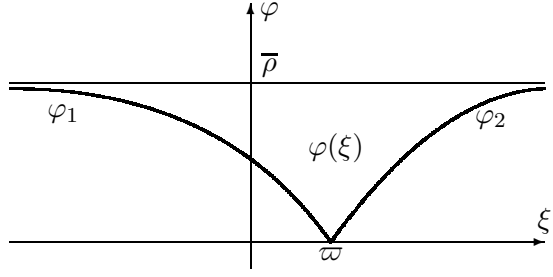
%%%%%%%%%%%%%%%%%%%%%%%%%%%%% End Figure pasting

Clearly $\phi$ is a classical solution for $\xi\ne\varpi$; however, because of the discontinuity of $\phi'$ at $\varpi$, notice that the pasting \eqref{e:phi} has possibly a meaning only if the pasting occurs at the point $(\varpi,0)$  in the $(\xi,\phi)$-plane and under \D{1} or \D2. Indeed, either in the case \D{0} or in the case of a pasting at a point $(\xi_0,\rho_0)$, with $\xi_0\in\R$ and $\rho_0\in(0,\orho]$, the term $D(\phi)\phi'$ produces a Dirac mass at $\xi_0$ because $D(\rho_0)>0$; this does not make $\phi$ a (weak) solution to \eqref{e:tws}.

We denote by $c_1^*$ the threshold introduced in Theorem \ref{t:fosemi} for profiles from $\orho$. By the proof of Theorem \ref{t:semi}, we deduce that $\phi_2(\xi) = \tilde\phi_1(-\xi)$, where $\tilde\phi_1$ is the profile from $\orho$ corresponding to speed $-c$ and flux $-f$. If we denote by $c_2^*$ the threshold analogous to $c_1^*$ but for profiles {\em to} $\orho$, then $c_2^*$ satisfies \eqref{e:cs} with $-c$ and $-h$ replacing $c$ and $h$, respectively.

\begin{proposition}
Assume either \D{1} or \D{2} and let $\phi$ be as in \eqref{e:phi}. Then $\phi$ is a solution to \eqref{e:tws} if and only if $c\in[c_1^*, -c_2^*]$. 
\end{proposition}
\begin{proof}
In order to prove that $\phi$ is a weak solution to \eqref{e:tws} we must verify Definition \ref{d:tws} when $I=\R$; indeed, we only have to focus on a neighborhood of the pasting point $\varpi$. Therefore, let $\psi\in C_0^\infty(\R)$ with $\psi(\varpi)\ne0$; without loss of generality we can assume that $\psi(\varpi)=1$. We split the integral
\begin{equation}\label{e:integral}
\int_\R \left\{\left(D\left(\phi(\xi)\right)\phi'(\xi) - f\left(\phi(\xi)\right) + c\phi(\xi) \right)\psi'(\xi) - g\left(\phi(\xi)\right)\psi(\xi)\right\}\,d\xi
\end{equation}
into two parts, integrating separately in $(-\infty, \varpi)$ and in $(\varpi,\infty)$. %We deal first with case \D{1}; case \D{2} is analogous and briefly discussed later on.

A simple integration by parts shows that
\begin{align*}
\int_{-\infty}^{\varpi} \left\{\left(D(\phi)\phi' - f(\phi) + c\phi \right)\psi' - g(\phi)\psi\right\}\,d\xi & = \lim_{\xi\to \varpi^-}D\left(\phi_1(\xi)\right)\phi_1^{\prime}(\xi)
\\
& = \lim_{\phi\to 0^-}z_1(\phi).%=z_1(0)<0.D(0)\phi_1^{\prime}(\varpi)=0.
\end{align*}
%If $c<c_1^*$ we have instead $\phi_1'(\varpi)=-\infty$ but
%\[
%\lim_{\xi\to \varpi^-}D\left(\phi_1(\xi)\right)\phi_1^{\prime}(\xi) = \lim_{\phi\to 0^-}z_1(\phi)=z_1(0)<0.
%\]
By Theorem \ref{t:semiequiv} we conclude
\[
\int_{-\infty}^{\varpi}\left\{\left(D\left(\phi\right)\phi' - f\left(\phi\right) + c\phi \right)\psi' - g\left(\phi\right)\psi\right\}\,d\xi =
\left\{
\begin{array}{cl}
0 & \text{ if }c\ge c_1^*,
\\
z_1(0)<0 & \text{ if }c< c_1^*.
\end{array}
\right.
\]
Now, we consider the integration in $(\varpi,\infty)$. We argue as above but also recall the proof of 
Theorem \ref{t:semi}, see what we pointed out just above the statement of this proposition. We deduce
\[
\int_{\varpi}^{\infty}\left\{\left(D\left(\phi\right)\phi' - f\left(\phi\right) + c\phi \right)\psi' - g\left(\phi\right)\psi\right\}\,d\xi =
\left\{
\begin{array}{cl}
0 & \text{ if }-c\ge c_2^*,
\\
-z_2(0)<0 & \text{ if }-c< c_2^*.
\end{array}
\right.
\]
Therefore, the integral \eqref{e:integral} vanishes if and only if $c_1^*\le c\le -c_2^*$.
\end{proof}

From the proof of the above proposition and \eqref{e:slope1}, \eqref{e:slope2}, we deduce that, in case \D{1}, the condition $c\in[c_1^*, -c_2^*]$ is equivalent to require that both $\phi_1'(\varpi^-)$ and $\phi_2'(\varpi^+)$ are real numbers.

The thresholds $c_1^*$ and $c_2^*$ have not an explicit expression but are estimated in \eqref{e:cs}. In order that there exists $c$ in the range $[c_1^*,-c_2^*]$ we need that $c_1^* + c_2^*\le0$. 
%To check this inequality we use the sufficient condition provided by  estimates \eqref{e:cs}; by Remark \ref{r:exchange}, we deduce that $c_1^* + c_2^*\le0$ if
%\[
%4\sqrt{\sup_{s\in (0, \overline \rho)}\frac{D(s)g(s)}{s}}
%\le
%-\max_{\rho \in [0, \overline \rho]}h(\rho) +\min_{\rho \in [0, \overline \rho]}h(\rho).
%\]
%Under (D) this inequality is never satisfied. 
However, by \eqref{e:cs} we see that
\[
c_1^*+c_2^*\ge 4\sqrt{\dot D(0)g(0)}\ge0,
\]
which shows that the {\em reverse} inequality holds. This leaves open only the eventuality
\begin{equation}\label{e:D2c1c2}
\text{condition \D{2} holds and $c_1^*=-c_2^*$.}
\end{equation}
In this case we are led to the unique choice $c=c_1^*=-c_2^*$. As we noted above, we cannot establish whether the case $c_1^*=-c_2^*$ can occur. Apart from this (possible) case, the construction in \eqref{e:phi} {\em never} leads to a solution of \eqref{d:tws}. In other words and apart from case \eqref{e:D2c1c2}: for any fixed $c$, if a semi-wavefront profile has finite slope when it reaches zero, then the other one has infinite slope. This is equivalent to say that $D(\phi)\phi'$ is discontinuous at $\varpi$ and then its derivative produces a Dirac mass at that point.

A comparison with the special dispersive equations considered in \cite{Lenells_Classification} is interesting. With reference to the Camassa-Holmes equation, the third-order equation for the profile is reduced to a second-order equation, which is somewhat analogous to \eqref{e:tws} with $D(\phi) = 2(\phi-c)$. For the corresponding profile, it is possible to prove that $2(\phi-c)\phi'\in W_{\rm loc}^{1,1}(\R)$, i.e. $2(\phi-c)\phi'$ is absolutely continuous \cite[Lemma 5]{Lenells_CH}; this makes possible the pasting.

%%%%%%%%%%%%%%%%%%%%%%%%%%%%%%%%%%%%%%%%%%%%%%%%%%%%%%%
%%%%%%%%%%%%%%%%%%%%%%%%%%%%%%%%%%%%%%%%%%%%%%%%%%%%%%%
%%%%%%%%%%%%%%%%%%%%%%%%%%%%%%%%%%%%%%%%%%%%%%%%%%%%%%%
\section{Diffusion with infinite slope at $0$}\label{sec:slopeinfty}
\setcounter{equation}{0}

In this last section we only require $D\in C[0, \orho]\cap C^1(0, \orho)$ and assume ({\^D}). This means that we allow $D$ to have infinite slope at $0$; the differentiability of $D$ at $\orho$ plays no role in the discussion below. Most of the previous results still hold under ({\^D}): indeed, the comparison-type techniques in Section \ref{s:comp} only depend on the continuity of $D$ and this is also the case for Lemma \ref{l:ab}, while  Proposition \ref{p:monot} simply involves the values of $\dot D$ in the open interval $ (0, \orho)$. As a consequence, we only need to focus on problem  \eqref{e:fo} and the equivalence discussed in Theorem \ref{t:semiequiv}.

\begin{proofof}{Theorem \ref{t:Dinf}} Fix $c\in \mathbb{R}$. The proof depends on the properties of $D$.

\smallskip
\noindent I. \emph{Assume condition {\rm ({\^D}${}_0$)}}. \, In this case it is possible to find real values $a_1, a_2$ and strictly positive numbers $b_1, b_2$ in such a way that, if we denote $D_i(\phi)=:a_i\phi+b_i$ for $i=1,2$, then $D_1(\phi)< D(\phi)<D_2(\phi)$ for $\phi\in [0, \orho]$. Problem \eqref{e:fo}, when replacing $D$ with $D_1$ and $D_2$, is uniquely solvable by Theorem \ref{t:fosemi}, because condition \D{0} holds for both $D_1$ and $D_2$. Let $z_1$ and $z_2$ be these solutions, respectively; see Figure \ref{f:manyzeta2}{\em (a)}. In particular, Theorem \ref{t:fosemi} implies $z_1(0)<0$. Notice that $z_1$ is a strict lower-solution and $z_2$ is a strict upper-solution of \eqref{e:zeq} on $[0, \orho)$; we claim that
\begin{equation}\label{e:cD12}
z_1(\phi) > z_2(\phi), \qquad \phi \in [0, \orho).
\end{equation}
Indeed, since $D_1<D_2$ in $[0, \orho]$ we deduce that $z_1$ is a strict lower-solution of
\begin{equation}\label{e:D2}
\dot z(\phi)=h(\phi)-c-\frac{D_2(\phi)g(\phi)}{z(\phi)}, \quad \phi \in [0, \orho).
\end{equation}
Let $\gamma(\phi)$ be the solution of \eqref{e:D2} satisfying $\gamma(0)=z_1(0)$ and assume that $\gamma$ is defined in $[0, \beta)$, with $\beta\le \orho$. By Lemma \ref{l:keyif}\emph{(1.i)} we have $\gamma(\phi)>z_1(\phi)$ for $\phi\in (0, \beta)$. Notice that
\begin{equation*}
\dot \gamma(\phi)-\dot z_1(\phi)=g(\phi)\left[ \frac{D_2(\phi)}{-\gamma(\phi)}-\frac{D_1(\phi)}{-z_1(\phi)}\right]>0, \quad \phi \in [0, \beta),
\end{equation*}
which makes impossible the case $\beta=\orho$. Hence $\beta<\orho$ and this implies $z_2(0)<z_1(0)$. Moreover, if there exists $\phi_0 \in (0, \orho)$ such that $z_1(\phi_0)=z_2(\phi_0)$, we deduce as above $z_2(\hat \phi)=0$ for some $\hat \phi <\orho$, i.e., a contradiction. Claim \eqref{e:cD12} is then proved. 

Since $z_1(\orho)=0$, we can find an increasing sequence $\{\psi_n\}\subset(0,\orho)$, which converges to $\orho$ and such that $\{z_1(\psi_n)\}$ is also increasing. Denote with $\zeta_n$ the solution of final-value problem
\begin{equation*}
\left\{
\begin{array}{l}
\dot z (\varphi)= h(\varphi)-c-\frac{D(\varphi)g(\varphi)}{z(\varphi)},\ \phi<\psi_n,\\
z(\psi_n)=z_1(\psi_n).
\end{array}
\right.
\end{equation*}
By means of Lemma \ref{l:iandfvp}{\sl (2)}, the solution $\zeta_n$ is unique and it is defined on $(0, \psi_n]$. Furthermore, the sequence $\left\{\zeta_n(\phi)\right\}_n$ is increasing for all $\phi$ and, by Lemma \ref{l:keyif}\emph{(2)}, it satisfies $z_2(\phi)< \zeta_n(\phi)<z_1(\phi)$ for $\phi \in (0, \psi_n)$. Since $\zeta_n$  is bounded away from $0$, we can extend it to $0$ by continuity. We define
\begin{equation*}
z(\phi)=\lim_{n \to \infty}\zeta_n(\phi), \qquad \phi \in [0, \orho).
\end{equation*}
As in the proof of Theorem \ref{t:fosemi}\emph{(b)}, we can prove that $z(\phi)$ is the required solution of problem \eqref{e:fo} with  $z(0)\le z_1(0)<0$.

%%%%%%%%%%%%%%%%%%%%%%%% Figure manyzeta2
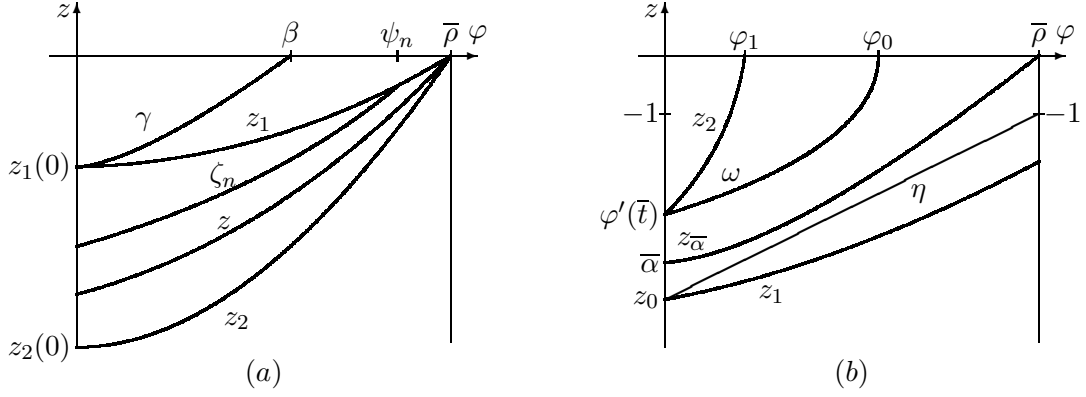
\begin{figure}[htbp]
\begin{picture}(100,140)(80,-60)
\setlength{\unitlength}{1pt}

\put(160,0){
%(a)
\put(-40,60){
\put(0,0){\vector(1,0){160}}
\put(160,8){\makebox(0,0){$\varphi$}}
\put(10,-110){\vector(0,1){130}}
\put(5,17){\makebox(0,0){$z$}}
\put(150,2){\line(0,-1){110}}
\put(150,8){\makebox(0,0){$\overline{\rho}$}}
\put(90,3){\makebox(0,0)[b]{$\beta$}}
\put(90,-2){\line(0,1){4}}
\put(130,3){\makebox(0,0)[b]{$\psi_n$}}
\put(130,-2){\line(0,1){4}}
\put(8,-42){\makebox(0,0)[r]{$z_1(0)$}}
\put(8,-110){\makebox(0,0)[r]{$z_2(0)$}}
\put(0,0){\thicklines{\qbezier(10,-42)(35,-40)(90,0)}} %gamma
\put(35,-25){\makebox(0,0){$\gamma$}}
\put(0,0){\thicklines{\qbezier(10,-42)(85,-40)(150,0)}} %z_1
\put(78,-25){\makebox(0,0){$z_1$}}

\put(0,0){\thicklines{\qbezier(10,-72)(85,-50)(130,-11.5)}} %\zeta_n
\put(65,-45){\makebox(0,0){$\zeta_n$}}

\put(0,0){\thicklines{\qbezier(10,-90)(85,-70)(150,0)}} %z
\put(65,-63){\makebox(0,0){$z$}}

\put(0,0){\thicklines{\qbezier(10,-110)(75,-110)(150,0)}} %z2
\put(70,-100){\makebox(0,0){$z_2$}}

\put(80,-120){\makebox(0,0){$(a)$}}
}

%(b)
\put(180,60){
\put(0,0){\vector(1,0){160}}
\put(160,8){\makebox(0,0){$\varphi$}}
\put(10,-110){\vector(0,1){130}}
\put(5,17){\makebox(0,0){$z$}}
\put(150,2){\line(0,-1){110}}
\put(150,8){\makebox(0,0){$\overline{\rho}$}}
\put(90,9){\makebox(0,0)[t]{$\phi_0$}}
\put(90,-2){\line(0,1){4}}
\put(40,9){\makebox(0,0)[t]{$\phi_1$}}
\put(40,-2){\line(0,1){4}}
%
%\put(8,-42){\makebox(0,0)[r]{$\omega(0)/2$}}
%
\put(8,-60){\makebox(0,0)[r]{$\phi'(\overline{t})$}}
\put(8,-78){\makebox(0,0)[r]{$\overline{\alpha}$}}
\put(8,-92){\makebox(0,0)[r]{$z_0$}}
%
%\put(8,-110){\makebox(0,0)[r]{$\eta(0)-1$}}
%
\put(8,-22){\makebox(0,0)[r]{$-1$}}
\put(8,-22){\line(1,0){4}}
\put(152,-22){\makebox(0,0)[l]{$-1$}}
\put(148,-22){\line(1,0){4}}
%
% curves
\put(150,-22){\thicklines{\line(-2,-1){140}}}%eta
\put(105,-52){\makebox(0,0){$\eta$}}

\put(0,0){\thicklines{\qbezier(10,-60)(35,-35)(40,0)}} %z2
\put(25,-25){\makebox(0,0){$z_2$}}

\put(0,0){\thicklines{\qbezier(10,-60)(90,-35)(90,0)}} %omega
\put(35,-45){\makebox(0,0){$\omega$}}

\put(0,0){\thicklines{\qbezier(10,-78)(55,-75)(150,0)}} %z2
\put(20,-70){\makebox(0,0){$z_{\overline{\alpha}}$}}

\put(0,0){\thicklines{\qbezier(10,-92)(75,-80)(150,-40)}} %z1
\put(50,-90){\makebox(0,0){$z_1$}}
\put(80,-120){\makebox(0,0){$(b)$}}
}
}
\end{picture}
\caption{\label{f:manyzeta2}{{\em (a)}: Case ({\^D}${}_0$). The lower-solution $z_1$, the upper-solution $z_2$ and the solution $z$. {\em (b):} Case ({\^D}${}_1$). The solutions $z_{\overline{\alpha}}$, $z_1$, $z_2$, the upper-solution $\eta$ and the lower-solution $\omega$; here, $z_0<-1$ satisfies \eqref{e:z0}.}}
\end{figure}
%%%%%%%%%%%%%%%%%%%%%%%%%%%%% End Figure manyzeta2

\medskip

\noindent II. \emph{Assume condition {\rm ({\^D}${}_1$)}}. \, The  proof splits into three parts.

\smallskip

\noindent{\emph{(a) Existence of a lower-solution.} }  We show that there exist $\phi_0\in (0, \orho)$ and a strict lower-solution $\omega \colon [0, \phi_0] \to \mathbb{R}$ for \eqref{e:zeq}, such that $\omega(\phi)<0$ for $\phi\in [0, \phi_0)$ and $\omega(\phi_0)=0$. This means that 
\begin{equation}\label{e:Dinftylower}
\dot \omega<h(\phi)-c-\frac{D(\phi)g(\phi)}{\omega(\phi)}, \quad \phi \in (0, \phi_0).
\end{equation}
Let $0<M<N$ be two constants such that $h(\phi)-c>-M$ for $\phi \in [0, \orho]$. By (g) and ({\^D}${}_1$) there is $\varepsilon >0$ such that $D(\phi)g(\phi)>\frac{N^2\phi}{4}$ for $\phi\in (0, \varepsilon]$.
Therefore we shall prove \eqref{e:Dinftylower} if we find $\phi_0$ and $\omega$ solving 
\begin{equation}\label{e:Dinftylowerl}
\dot \omega=-M-\frac{N^2\phi}{4\omega(\phi)}, \quad \phi \in (0, \phi_0).
\end{equation}
It is not easy to solve directly this equation; so, we exploit the second-order equation which corresponds to it, in the same way that \eqref{e:tws} corresponds to \eqref{e:zeq}. 

Consider the equation $u^{\prime \prime}+Mu^{\prime}+\frac{N^2}{4}u=0$
and the solution
\begin{equation*}
\phi(t)=\eps\mbox{e}^{-\frac{Mt}{2}}\left( \cos(\alpha t) +\frac{M}{2\alpha}\sin (\alpha t)\right), \qquad \alpha=\frac{\sqrt{N^2-M^2}}{2}.
\end{equation*}
We denote $\overline t=\frac{1}{\alpha}\left[ \mbox{arctg}(-\frac{2\alpha}{M})+\pi\right]$ and notice that
\begin{equation*}
\phi^{\prime}(t)=-\eps\mbox{e}^{-\frac{Mt}{2}}\left(\frac{M^2}{4\alpha}+\alpha \right)\sin (\alpha t), \quad t\in \mathbb{R}.
\end{equation*} 
We have that $\phi(0)=\eps$ and $\phi(\overline t)=0$, $\phi(t)$ is positive and decreasing in $ [0,\overline t)$, $\phi'(0)=0$. Hence, the function $\phi$ is invertible and we denote by $t=t(\phi)$, $\phi\in[0, \eps]$, its inverse function. If we define $\omega(\phi):=\phi^{\prime}\left(t(\phi)\right)$ for $\phi\in[0, \eps]$ and  $\phi_0:=\eps$, see Figure \ref{f:manyzeta2}{\em (b)}, then it is not difficult to show that $\omega (\phi)$ is a solution of \eqref{e:Dinftylowerl}. Our claim is proved.
%\begin{equation*}
%\dot \omega=-M-\frac{N^2\phi}{4\omega(\phi)}, \quad \phi \in (0, c_1);
%\end{equation*}
%we also have that
%\begin{equation*}
%\dot \omega<h(\phi)-c-\frac{D(\phi)g(\phi)}{\omega(\phi)}, \quad \phi \in (0, c_1),
%\end{equation*}
%hence $\omega (\phi)$ is a lower-solution of \eqref{e:zeq} in $(0, c_1)$.

\medskip

\noindent{\emph{(b) Solution of problem \eqref{e:fo}.} } Consider the linear function $\eta(\phi)$ defined in \eqref{e:etaline} with  $z_0 <\phi^{\prime}(\overline t)$, see Figure \ref{f:manyzeta2}. We showed in the proof of Theorem \ref{t:fosemi}, part \emph{(a)}, that $\eta(\phi)$ is a strict upper-solution of \eqref{e:zeq} in $[0, \orho]$; the proof does not depend on $\dot D(0)$. By Lemma \ref{l:keyif}\emph{(1.ii)}, the solution $z_1$ of the initial-value problem
\begin{equation*}
\left\{
\begin{array}{ll}
\dot z=h(\phi)-c-\frac{D(\phi)g(\phi)}{z(\phi)},\\
z(0)=\eta(0),
\end{array}
\right.
\end{equation*}
satisfies $z_1(\phi)<\eta(\phi)$ for $\phi \in (0, \orho)$. In particular $z_1(\orho) \le -1$. Similarly, by Lemma \ref{l:keyif}\emph{(1.i)}, the solution $z_2$
of the initial-value problem
\begin{equation*}
\left\{
\begin{array}{ll}
\dot z=h(\phi)-c-\frac{D(\phi)g(\phi)}{z(\phi)},\\
z(0)=\omega(0),
\end{array}
\right.
\end{equation*}
with $\omega(\phi)$ defined in step \emph{(a)}, satisfies $z_2(\phi)>\omega(\phi)$ for $\phi\in(0,\phi_1)$, where $[0,\phi_1)$ is the maximal-existence interval of $z_2$. This implies that $[0, \phi_1)\subset[0, c_1)$.  

Now, consider the family $z_{\alpha}$ of solutions of \eqref{e:zeq} with $z_{\alpha}(0)=\alpha$, for $\alpha \in [z_1(0), z_2(0)]$, and apply a shooting argument. It is not difficult to find $\overline \alpha \in \left(z_1(0), z_2(0)\right)$ such that the corresponding function $z_{\overline \alpha}$ is a solution of problem \eqref{e:fo}, hence with  $z_{\overline \alpha}(0)<z_2(0)<0$.

\medskip

\noindent{\emph{(c) Uniqueness.} } The reasoning in the proof of Theorem \ref{t:fosemi}\emph{(c)} applies also here.

\medskip

\noindent III. The equivalence between semi-wavefront solutions $\phi$ and solutions $z$ of \eqref{e:fo}  can be proved as in Theorem \ref{t:semiequiv}. In particular, with reference to that proof, in case {\rm ({\^D}${}_0$)} inequality \eqref{e:z(0)D0} still holds, while in case {\rm ({\^D}${}_1$)} we are in case {\em(a)} because $z(0)<0$ for every $c$. This proves \eqref{e:lastissimo}.
\end{proofof}

We also point out that the negative results of Section \ref{sec:pasting}, concerning the impossibility of pasting semi-wavefronts, still hold under ({\^D}${}$).
%%%
\section*{Acknowledgment}
We thank M.D. Rosini for drawing our attention on equation \eqref{e:E} as a model for crowd dynamics in \cite{BTTV}. The first author also thanks A. Bressan and A. Novikov for useful discussions on this paper.
The authors are members of the {\em Gruppo Nazionale per l'Analisi Matematica, la Probabilità e le loro Applicazioni (GNAMPA)} of the {\em Istituto Nazionale di Alta Matematica (INdAM)}. The first author was supported by the project {\em Balance Laws in the Modeling of Physical, Biological and Industrial Processes} of GNAMPA.

%%%%%%%%%%%%%%%%%%%%%%%%%%%%%%%%%%%%%%%%%%%%%%%%%%%%%%%%%%%%%%%%%

\end{document}